\pdfoutput=1
\documentclass[12pt]{amsart}
\allowdisplaybreaks
\usepackage{amssymb}

\usepackage[utf8]{inputenc}
\usepackage{graphicx}
\usepackage[svgnames]{xcolor}
\usepackage{fullpage,microtype}
\usepackage{array,etoolbox,aliascnt,comment}
\usepackage{amsmath,amsthm,mathtools,stmaryrd,bm,tikz-cd}
\usepackage[shortlabels]{enumitem}

\usepackage[unicode,bookmarksnumbered,bookmarksopen,psdextra,colorlinks=true,allcolors=DarkBlue]{hyperref}
\usepackage{bookmark}
\usepackage[capitalize]{cleveref}
\crefformat{footnote}{#2\footnotemark[#1]#3}

\usepackage{pgfplots}
\pgfplotsset{compat=1.10}
\usepgfplotslibrary{fillbetween}

\usepackage{float}

\usepackage{standalone}
\usepackage{caption}
\usepackage[ligature]{semantic}

\usetikzlibrary{shapes.geometric,calc,fit,patterns}
\tikzset{
    dot/.style={circle,fill,inner sep=3\pgflinewidth},
    math node/.style={
        execute at begin node=$,
        execute at end node=$,
    },
}
\tikzset{
		big dot/.style={
			circle, inner sep=0pt, 
			minimum size=1mm, fill=black
		}
	}
\tikzset{
		Big dot/.style={
			circle, inner sep=0pt, 
			minimum size=2.2mm, fill=black
		}
	}

\makeatletter

\mathlig{->}{\to}
\mathlig{-->}{\longrightarrow}
\mathlig{|->}{\mapsto}
\mathlig{|-->}{\longmapsto}
\mathlig{`->}{\hookrightarrow}
\mathlig{`-->}{\longhookrightarrow}
\mathlig{->>}{\twoheadrightarrow}
\mathlig{-->>}{\longtwoheadrightarrow}
\mathlig{<-}{\leftarrow}
\mathlig{<--}{\longleftarrow}
\mathlig{<-|}{\mapsfrom}
\mathlig{<--|}{\longmapsfrom}
\mathlig{<<-}{\twoheadleftarrow}
\mathlig{<<--}{\longtwoheadleftarrow}
\mathlig{<->}{\leftrightarrow}
\mathlig{<-->}{\longleftrightarrow}
\mathlig{--}{\relbar\joinrel\relbar}

\def\coloniff{\DOTSB\;:\Longleftrightarrow\;}

\newcommand*\hash{}
\let\hash=\#
\newcommand*\atsf{}
\let\atsf=\@
\newcommand*\amp{}
\let\amp=\&
\newcommand*\nmsp{} 
\let\nmsp=\!
\newcommand*\dollar{}
\let\dollar=\$
\protected\def\#{\ifmmode\expandafter\mathbb\else\expandafter\hash\fi}
\protected\def\*#1{{\text{\boldmath$\mathbf{\newmcodes@#1}$}}}
\protected\def\@{\ifmmode\expandafter\mathcal\else\expandafter\atsf\fi}
\protected\def\&{\ifmmode\expandafter\mathfrak\else\expandafter\amp\fi}
\protected\def\!#1{\mathsf{\newmcodes@#1}}
\protected\def\${\ifmmode\expandafter\mathnormal\else\expandafter\dollar\fi}

\newcommand*\circumflex{}
\let\circumflex=\^
\newcommand*\tildeaccent{}
\let\tildeaccent=\~
\newcommand*\hyphen{}
\let\hyphen=\-
\newcommand*\underscore{}
\let\underscore=\_

\newcommand*\dotaccent{}
\let\dotaccent=\.
\protected\def\^{\ifmmode\expandafter\widehat@\else\expandafter\circumflex\fi}
\protected\def\widehat@#1{\binrel@{#1}\binrel@@{\widehat{#1}}}
\protected\def\~{\ifmmode\expandafter\widetilde@\else\expandafter\tildeaccent\fi}
\protected\def\widetilde@#1{\binrel@{#1}\binrel@@{\widetilde{#1}}}
\protected\def\-{\ifmmode\expandafter\overline@\else\expandafter\hyphen\fi}
\protected\def\overline@#1{\binrel@{#1}\binrel@@{\overline{#1}}}
\protected\def\_{\ifmmode\expandafter\underline@\else\expandafter\underscore\fi}
\protected\def\underline@#1{\binrel@{#1}\binrel@@{\underline{#1}}}
\protected\def\.{\ifmmode\expandafter\dot@\else\expandafter\dotaccent\fi}
\protected\def\dot@#1{\binrel@{#1}\binrel@@{\dot{#1}}}

\let\emptyset=\varnothing

\let\epsilon=\varepsilon
\let\phi=\varphi

\def\twoheadrightarrow{\mathrel{\mathrlap{\to}\,{\to}}}
\def\longtwoheadrightarrow{\relbar\joinrel\twoheadrightarrow}
\def\AND{\mathrel{\,\amp\,}}
\def\OR{\mathrel{\,\text{or}\,}}

\DeclarePairedDelimiter\paren{\lparen}{\rparen}

\DeclarePairedDelimiterX\set[2]{\{}{\}}{#1\nonscript\;\delimsize\vert\nonscript\;#2}

\DeclarePairedDelimiter\abs{\lvert}{\rvert}
\DeclarePairedDelimiter\norm{\lVert}{\rVert}

\numberwithin{equation}{section}
\crefname{equation}{}{}
\crefname{enumi}{}{}
\crefname{enumii}{}{}
\crefname{enumiii}{}{}
\crefname{figure}{Figure}{Figures}

\setlist{font=\textup}
\crefdefaultlabelformat{\textup{#2#1#3}}

\newtheorem{theorem}[equation]{Theorem}

\newtheorem{lemma}[equation]{Lemma}
\newtheorem{proposition}[equation]{Proposition}
\newtheorem{corollary}[equation]{Corollary}

\newtheorem*{theorem*}{Theorem}
\newtheorem*{fact*}{Fact}
\newtheorem*{lemma*}{Lemma}
\newtheorem*{proposition*}{Proposition}
\newtheorem*{corollary*}{Corollary}
\newtheorem*{conjecture*}{Conjecture}

\theoremstyle{definition}
\newtheorem{observation}[equation]{Observation}
\newtheorem{definition}[equation]{Definition}
\newtheorem{remark}[equation]{Remark}
\newtheorem{example}[equation]{Example}

\newtheorem{question}[equation]{Question}
\crefname{question}{Question}{Questions}
\newtheorem{hypothesis}[equation]{Hypothesis}

\newtheorem*{definition*}{Definition}
\newtheorem*{remark*}{Remark}
\newtheorem*{example*}{Example}
\newtheorem*{problem*}{Problem}
\newtheorem*{question*}{Question}
\newtheorem*{hypothesis*}{Hypothesis}
\newtheorem*{convention*}{Convention}
\newtheorem*{addendum*}{Addendum}

\let\defn=\textbf

\newcommand{\w}{\mathbf{w}}
\newcommand{\bfP}{\mathbf{P}}
\newcommand{\ww}[1]{{ \pgothfamily w\normalfont}}

\newcommand*\SL{\mathrm{SL}}

\DeclareMathOperator{\RandomMF}{FMaxSF}
\DeclareMathOperator{\FMSF}{FMSF}
\DeclareMathOperator{\WMSF}{WMSF}

\DeclareMathOperator{\Aut}{Aut}

\def\acts{\curvearrowright}

\newcommand*\cl{\mathrm{cl}}
\newcommand*\bs{{\setminus}}

\newcommand{\ledge}{<}
\newcommand{\lweight}[1]{\ledge_{#1}}
\newcommand{\lw}{\lweight{\w}}
\newcommand{\gedge}{>}
\newcommand{\gweight}[1]{\gedge_{#1}}
\newcommand{\gw}{\gweight{\w}}

\newlist{eqenum}{enumerate}{1}
\let\c@eqenumi=\@undefined
\newaliascnt{eqenumi}{equation}
\setlist[eqenum]{label=(\arabic*),align=left,labelindent=0pt,labelwidth=!,resume}
\pretocmd\eqenum{\edef\enit@resume@eqenum{\noexpand\c@eqenumi\the\c@eqenumi}}{}{}
\crefname{eqenumi}{}{}

\pgfdeclarelayer{back}
\pgfdeclarelayer{pre main}
\pgfdeclarelayer{front}
\pgfsetlayers{back,pre main,main,front}

\newcommand{\tocdotfill}{%
    \leaders\hbox to .44em{\hss.\hss}\hfill
}

\def\l@section{\@tocline{1}{5pt}{0pc}{}{}}
\renewcommand{\tocsection}[3]{%
	\indentlabel{\@ifnotempty{#2}{\makebox[20pt][l]{%
				\ignorespaces#1 #2.\hfill}}}{\sc #3}\tocdotfill}

\newdimen{\tocsubsecmarg}
\def\l@subsection{\@tocline{2}{3pt}{0pc}{\tocsubsecmarg}{}}
\renewcommand{\tocsubsection}[3]{%
	\indentlabel{\@ifnotempty{#2}{\makebox[30pt][l]{%
				\ignorespaces#1 #2.\hfill}}}{#3}\tocdotfill}
\makeatother
\let\oldtocsubsection=\tocsubsection
\renewcommand{\tocsubsection}[2]{\hspace{3em} \oldtocsubsection{#1}{#2}}

\makeatother

\ExplSyntaxOn
\makeatletter
\tl_clear:N \__firstaid_cref_override_label
\tl_clear:N \__firstaid_cref_override_counter
\tl_clear:N \__firstaid_cref_override_type
\cs_if_exist:NTF \firstaid@cref@updatelabeldata {
    \cs_gset:Npn \firstaid@cref@updatelabeldata #1 {
        \cref@constructprefix{#1}{\cref@result}%
        \tl_set:Ne \@tempa {
            \bool_lazy_and:nnT
                {\str_if_eq_p:VV \@currentlabel \__firstaid_cref_override_label}
                {\str_if_eq_p:VV \@currentcounter \__firstaid_cref_override_counter}
                {\tl_if_empty:NF \__firstaid_cref_override_type
                    {\use_i:nnn \__firstaid_cref_override_type}}
            \use:n {#1}
        }
        \@ifundefined{cref@ \@tempa @alias}%
            {}%
            {\exp_args:NNo \def\@tempa{\csname cref@ \@tempa @alias\endcsname}}%
        \protected@xdef\cref@gcurrentlabel@temp{%
            [\@tempa][\arabic{#1}][\cref@result]%
            \csname p@#1\expandafter\endcsname\csname the#1\endcsname}%
        \aftergroup\firstaid@cref@smugglelabel
    }
} {
    \PackageWarning {firstaid2} {
        Failed~to~patch~ \token_to_str:N \firstaid@cref@updatelabeldata .
    }
}
\cs_if_exist:NTF \cref@old@refstepcounter {
    \cs_gset_eq:NN \refstepcounter@noarg \cref@old@refstepcounter
    \cs_gset:Npn \refstepcounter@optarg [#1] #2 {
        \cref@old@refstepcounter {#2}
        \tl_set_eq:NN \__firstaid_cref_override_label \@currentlabel
        \tl_set_eq:NN \__firstaid_cref_override_counter \@currentcounter
        \tl_set:Nn \__firstaid_cref_override_type {#1}
    }
} {
    \PackageWarning {firstaid2} {
        Failed~to~patch~ \token_to_str:N \refstepcounter .
    }
}
\cs_if_exist:NTF \cref@old@makefntext {
    \cs_gset:Npn \@makefntext {
        \tl_set:Nn \@currentcounter {footnote}
        \cref@old@makefntext
    }
} {
    \PackageWarning {firstaid2} {
        Failed~to~patch~ \token_to_str:N \@makefntext .
    }
}
\makeatother
\ExplSyntaxOff

\begin{document}

\title{Nonamenable subforests of multi-ended quasi-pmp graphs}

\begin{abstract}
We prove the a.e.\ nonamenability of locally finite quasi-pmp Borel graphs whose every component admits at least three nonvanishing ends with respect to the underlying Radon--Nikodym cocycle.
We witness their nonamenability by constructing Borel subforests with at least three nonvanishing ends per component, and then applying Tserunyan and Tucker-Drob's recent characterization of amenability for acyclic quasi-pmp Borel graphs.
Our main technique is a weighted cycle-cutting algorithm, which yields a weight-maximal spanning forest.
We also introduce a random version of this forest, which generalizes the Free Minimal Spanning Forest, to capture nonunimodularity in the context of percolation theory.
\end{abstract}

\author{Ruiyuan Chen}
\address{
Ruiyuan Chen \\
Faculty of Mathematics, Informatics, and Mechanics \\
University of Warsaw \\
Warsaw \\
Poland
}
\email{ruiyuan@mimuw.edu.pl}
\urladdr{
\url{https://rynchn.github.io/math}
}

\author{Grigory Terlov}
\address{
Grigory Terlov \\
Department of Statistics and Operations Research\\
University of North Carolina\\
Chapel Hill, NC \\
USA
}
\email{gterlov@unc.edu}
\urladdr{
\url{https://sites.google.com/view/gterlov/home}
}

\author{Anush Tserunyan}
\address{
Anush Tserunyan \\
Mathematics and Statistics Department \\
McGill University \\
Montr\'{e}al, QC \\
Canada
}
\email{anush.tserunyan@mcgill.ca}
\urladdr{
\url{https://www.math.mcgill.ca/atserunyan}
}





\maketitle
\tableofcontents


\section{Introduction}
\label{sec:intro}

Locally countable Borel graphs\footnote{\label{ft:Borel_graph}A \defn{Borel graph} on a standard Borel space $X$ is a graph with vertex set $X$ and edge set a symmetric Borel subset of $X^2$.} on standard probability spaces have been a center of attention in a variety of areas such as Descriptive Graph Combinatorics, Measured Group Theory, Ergodic Theory, and the study of Countable Borel Equivalence Relations\footnote{\label{ft:cber-fber}An equivalence relation $E$ on a standard Borel space $X$ is \defn{countable} (resp.\ \defn{finite}) \defn{Borel} if it is Borel as a subset of $X^2$ and each equivalence class is countable (resp.\ finite).} (\defn{CBERs}). 
The interest in them is largely facilitated by the fact that these graphs arise as Schreier graphs of measurable actions of countable groups, enabling graph-theoretic and algorithmic approaches, as well as techniques from probabilistic combinatorics and geometric group theory, for studying such actions and their orbit equivalence relations. 

The majority of the developed theory concerns probability measure preserving (\defn{pmp}) actions, where in particular, notions like amenability are well-understood. 
On the other hand, much less is known about \defn{quasi-pmp} (nonsingular) actions, where points in the same orbit have different ``relative weights'', whence even free such actions may not reflect the properties of the acting group. 
For instance, the typical nonamenable group $\#F_2$ has amenable free quasi-pmp actions (see \cref{F2actsamenable}). 
In this article -- in the spirit of the Day--von Neumann question -- we construct nonamenable subforests of certain locally finite quasi-pmp Borel graphs, in particular witnessing their nonamenability. 
We do so using the interplay between the geometry of these graphs (the space of ends) and the behavior of the underlying Radon--Nikodym cocycle, which we interpret as a relative weight function on the graph that accounts for the failure of invariance of the underlying measure.

Just like amenability is an important conceptual threshold for groups, \defn{$\mu$-amenability} is an equally fundamental concept for CBERs/group actions/graphs on a standard probability space $(X,\mu)$. The concept of $\mu$-amenability was originally introduced in \cite{Zimmer:amenable_erg_actions}, and a very useful equivalent definition was given in \cite{Connes-Feldman-Weiss}; see also \cite[2.4]{JKLcber}. By the Connes--Feldman--Weiss theorem \cite{Connes-Feldman-Weiss}, $\mu$-amenability of CBERs is equivalent to \defn{$\mu$-hyperfiniteness}, i.e., being a countable increasing union of finite Borel equivalence relations\cref{ft:cber-fber} off of a $\mu$-null set. In the present paper, we are only concerned with the measured context, so we use the terms \textbf{amenable} and \textbf{hyperfinite} interchangeably, dropping $\mu$ from notation when it is clear from the context. We also use these terms for Borel actions of countable groups and for locally countable Borel graphs on $(X,\mu)$, calling them amenable or hyperfinite, when the induced orbit equivalence relation or connectedness relation, respectively, are so.



\subsection{Overview of nonamenability in the pmp setting}

A fundamental question in this line of research is to determine whether a given Borel action of a countable group on $(X,\mu)$ is amenable. 
By \cite[2.5(i)]{JKLcber}, every such action is amenable when the group is amenable. 
For free probability-measure-preserving (\defn{pmp}) actions, we also have the converse: the action is nonamenable when the group is nonamenable \cite[2.5(ii)]{JKLcber}. 
More generally, free pmp actions of countable groups tend to reflect the properties of the group. 
A big part of this is due to the theory of cost for pmp CBERs\footnote{\label{ft:pmp-quasi-pmp_CBER}A CBER $E$ on a standard probability space $(X,\mu)$ is \defn{pmp} (resp. \defn{quasi-pmp}) if every Borel automorphism $\gamma$ on $X$, that maps every point to an $E$-equivalent point, preserves $\mu$ (resp. $\mu$-null sets).} -- an analogue for CBERs of the free rank for groups -- which is only available in the pmp setting.
The \defn{cost} of a pmp CBER $E$ is defined as the infimum of the cost (i.e.\ half of expected degree) of its \defn{graphings}, i.e.\ Borel graphs $G$ on $X$ whose connectedness relation $\#E_G$ is equal to $E$ a.e.

Fundamental results of Gaboriau \cite{Gaboriau:mercuriale} and Hjorth \cite{Hjorth:cost_lemma} establish a strong analogy between free groups and \defn{treeable} pmp CBERs (those which admit acyclic graphings). 
In particular, an ergodic\footnote{An equivalence relation is \defn{ergodic} if every invariant measurable set is either null or conull.} pmp treeable CBER is nonamenable exactly when its cost is $>1$. 
By \cite{Gaboriau:mercuriale}, any acyclic graphing achieves the cost of a pmp CBER, so in the pmp context, acyclic Borel graphs of cost $> 1$ play the same role as free groups of rank $> 1$ among groups. 
In the spirit of the Day--von Neumann question as to whether every nonamenable group contains a copy of $\#F_2$ (the free group on $2$ generators), one tries to detect the nonamenability of a given pmp graph $G$ by exhibiting a nonamenable acyclic Borel subforest of $G$, or at least of its connectedness relation $\#E_G$. 
The following is a corollary of a theorem of Ghys \cite{Ghys:Stallings}, due to Gaboriau \cite[IV.24]{Gaboriau:cout}:

\begin{theorem}[Gaboriau--Ghys; 2000]\label{Gaboriau-Ghys}
The connectedness relation $\#E_G$ of an ergodic pmp graph $G$ whose a.e.~component has $\ge 3$ ends is of cost $>1$, hence $\#E_G$ is nonamenable. In fact, there is an ergodic subforest $F \subseteq \#E_G$ of cost $>1$ witnessing the nonamenability of $\#E_G$.
\end{theorem}
%

This result is a key ingredient in one of the proofs of the celebrated Gaboriau--Lyons theorem \cite{Gaboriau-Lyons}, which gives a positive answer to the Day--von Neumann question for Bernoulli shifts: if $\Gamma$ is a nonamenable countable group, then the orbit equivalence relation of its shift action on $([0,1]^\Gamma, \lambda^\Gamma)$, where $\lambda$ is the Lebesgue measure, admits an ergodic subequivalence relation induced by a free action of $\#F_2$. Indeed, in this proof, percolation theory yields a subgraph of the orbit equivalence relation to which \cref{Gaboriau-Ghys} applies, yielding an ergodic subforest of cost $> 1$, which is then upgraded to a free action of $\#F_2$ by \cite{Hjorth:cost_lemma}.

\subsection{Nonamenability in the quasi-pmp setting and our main result}

As for general CBERs on standard probability spaces, much less is known. By an argument of Kechris and Woodin, see \cite[Proposition 2.1]{Miller:thesis} or \cite[2.2]{Ts-Zomback:backward_ergodic}, these equivalence relations are quasi-pmp\cref{ft:pmp-quasi-pmp_CBER} after discarding a null set. However, most of the aforementioned theory of pmp actions and CBERs fails in the quasi-pmp setting, starting with the statement that nonamenability of the group implies nonamenability of the orbit equivalence relations of its free actions. 

\begin{example}\label{F2actsamenable}
Consider the action of $\#F_2$ on its boundary $\partial \#F_2$, which we identify with the set of infinite reduced words in the symmetric set $\{a^{\pm 1}, b^{\pm 1}\}$ of generators of $\#F_2$. This action is free except at countably many points (which we discard), and it is (Borel) hyperfinite, by \cite[8.2]{Dougherty-Jackson-Kechris}, because its orbit equivalence relation is equal to that induced by the one-sided shift map on $\partial \#F_2$. This implies that there is no invariant probability measure on $\partial \#F_2$, but there are certainly many quasi-invariant probability measures, e.g.\ the one with value $\tfrac{1}{4} \cdot (\tfrac{1}{3})^{n-1}$ on each cylindrical set based on a word of length $n$.
\end{example}

That properties of the group are not reflected by its free quasi-pmp actions is due to the fact that in the latter setting, points in the same orbit have different ``relative weights.'' This is made precise by the \defn{Radon--Nikodym cocycle} $(x,y) \mapsto \w^y(x) : E \to \#R^+$ of the orbit equivalence relation $E$ with respect to the underlying probability measure $\mu$, as defined in \cite[Section 8]{KMtopics}. In the present paper, we think of $\w^y(x)$ as the weight of $x$ relative to $y$, so we call cocycles to $\#R^+$ \textbf{relative weight functions}, hence the notation $\w$. The Radon--Nikodym cocycle ``corrects'' the failure of invariance of the measure $\mu$, enabling the (\defn{measurable}) \defn{mass transport principle}: for each $f : E \to [0, \infty]$,
\begin{equation}\label{eq:MMTP}
\int \sum_{y \in [x]_E} f(x,y) d\mu(x) = \int \sum_{x \in [y]_E} f(x,y) \w^y(x) d\mu(y). 
\end{equation}

\textit{In the absence of the theory of cost in the quasi-pmp setting, we look at the geometry of quasi-pmp graphs and the behaviour of the Radon--Nikodym cocycle along them.}
One has to first understand amenability for the simplest class of quasi-pmp graphs, namely, acyclic ones. This is done by Tserunyan and Tucker-Drob in a recent paper \cite{Ts-TuckerDrob:hyperfinite_in_tree}. 
To present their characterization result in analogy with the pmp setting, we first note that for acyclic ergodic pmp graphs, we can replace cost by geometry, namely, the cost of such a graph is $\le 1$ if and only if it has $\le 2$ ends a.e. 
Thus, an acyclic ergodic pmp graph is amenable exactly when it has $\le 2$ ends a.e.\ (proven directly by Adams in \cite{Adams:trees_amenability}). 
In \cite{Ts-TuckerDrob:hyperfinite_in_tree}, this is generalized to the quasi-pmp setting as follows: 

\begin{theorem}[Tserunyan--Tucker-Drob; 2025]\label{Anush-Robin}
An acyclic quasi-pmp Borel graph $G$ is amenable exactly when a.e.\ $G$-component has $\le 2$ $\w$-nonvanishing ends, where $\w$ is the Radon--Nikodym cocycle of $\#E_G$ with respect to the underlying measure.
\end{theorem}

\noindent Here, an end $\eta$ of $G$ is said to be \textbf{$\w$-vanishing} if the cocycle $\w$ converges to $0$ along any sequence $(x_n)$ of vertices converging to $\eta$, i.e., $\lim_n \w^{x_0}(x_n) = 0$. Notice that in \cref{F2actsamenable}, each connected component of the canonical Schreier graph has exactly one $\w$-nonvanishing end, namely, the forward direction of the one-sided shift map on $\partial \#F_2$.

\cref{Anush-Robin} indeed generalizes the pmp situation because $\w \equiv 1$ in that case, so every end is nonvanishing.
It is not hard to derive from \cref{Anush-Robin} that every \emph{locally finite} quasi-pmp graph $G$ that has $\ge 3$ $\w$-nonvanishing ends in a.e.\ $G$-component is nowhere amenable, see \cref{3-nonvanishing=>nonamenability}.
The main result of the present paper is an explicit construction of a subforest witnessing the nonamenability of $G$, thereby completing the generalization of \cref{Gaboriau-Ghys} to the quasi-pmp setting for locally finite graphs:

\begin{theorem}[see \cref{thm:mf-qpmp}]\label{intro:witnessing_nonamenability}
Let $G$ be a locally finite quasi-pmp Borel graph and let $\w$ denote its Radon–Nikodym cocycle with respect to the underlying probability measure.
If a.e.\ $G$-component has $\ge 3$ $\w$-nonvanishing ends, then there is a Borel subforest $F \subseteq G$ that is ergodic relative%
\footnote{\label{ftnote:relative_ergodicity} For a Borel graph $G$ on a measure space, a Borel subgraph $F \subseteq G$ is \defn{ergodic relative to $G$} if every $\#E_F$-invariant Borel set is $\#E_G$-invariant mod null.}
to $G$ such that for a.e.\ $F$-component, the space of $\w$-nonvanishing ends of that component is nonempty and perfect.\footnote{\label{ftnote:perfect_space} A topological space is \defn{perfect} if it has no isolated points, e.g., $\#Q$.}
In particular, by \cref{Anush-Robin}, $F$, and hence $G$, is nowhere amenable
\unskip.
\end{theorem}

\begin{remark}\label{exactly_2_ends}
For a locally finite quasi-pmp Borel graph $G$, if each $G$-component has \emph{exactly} two $\w$-nonvanishing ends, then $G$ is hyperfinite by \cite[5.1]{Mends}, or by a simple geometric argument using a maximal disjoint set of $\w$-bifurcations. 
As in the pmp case, we are unable to conclude anything if each component of $G$ has \emph{exactly} one $\w$-nonvanishing end.
Finally, if each $G$-component has zero $\w$-nonvanishing ends, then $G$ is smooth because local finiteness implies that there are only finitely many $\w$-maximal elements in each $G$-component.

\end{remark}

\begin{remark}
In the case of locally finite Borel graphs \cref{intro:witnessing_nonamenability} is a strengthening of \cref{Gaboriau-Ghys} even for pmp graphs because our ergodic forest $F$ is a subgraph of $G$ and not just of $\#E_G$. However, this is only due to the fact that we now know, by a theorem of Tserunyan \cite{Ts:hyperfinite_ergodic_subgraph} (which generalizes the analogous theorem of Tucker-Drob for pmp graphs), that every quasi-pmp ergodic graph admits an ergodic hyperfinite subgraph.
\end{remark}

\begin{remark}\label{cycle-cutting_algorithm}
A significant part of the proof of \cref{Gaboriau-Ghys}, namely that by Ghys, involves an analogue for pmp graphs of the Stallings analysis of ends of groups \cite{Stallings:ends}. In contrast, our subforest $F$ in \cref{intro:witnessing_nonamenability} is constructed via a much simpler cycle-cutting algorithm, which runs simultaneously on all $G$-components and cuts the $\w$-lightest edge in each simple cycle, using a fixed Borel linear ordering on edges as a tiebreaker.
\end{remark}

In \cref{sec:appl}, we give concrete applications of \cref{intro:witnessing_nonamenability} to \textbf{coinduced group actions} (\cref{coinduced-action_nonamenable}) and \textbf{cluster graphings} of nonunimodular graphs (\cref{cluster-graphing_nonamenable,ex:GP}).

\subsection{Application to percolation theory: Free $\w$-Maximal Spanning Forest}

Besides \cref{intro:witnessing_nonamenability}, our cycle-cutting algorithm described in \cref{cycle-cutting_algorithm} has other applications, in particular to \textbf{random forests} in probability theory. 
Indeed, this algorithm works abstractly on any countable graph $G$ equipped with a relative weight function $\w$ and a linear ordering $<$ (tiebreaker) on the edges of $G$, yielding what we call the \textbf{$\w$-maximal subforest} of $G$ (with respect to the tiebreaker $<$). 
This is a generalization to the (relatively) weighted setting of the \textbf{minimal subforest} algorithm used in probability, which simply deletes the $<$-least edge in each cycle, regardless of its $\w$-weight. 
In particular, just like the minimal subforest splits each cluster (i.e.\ connected component) of $G$ into infinite trees, the $\w$-maximal subforest does the same with \defn{$\w$-infinite} clusters (also known as \textbf{heavy} clusters), i.e.\ those whose total $\w$-weight is infinite:

\begin{proposition}\label{intro:heavy_splits_into_heavy}
Let $G$ be a graph with a relative weight function $\w$ and a linear ordering $<$ on the edges. Every $\w$-infinite component of $G$ splits into $\w$-infinite trees in the $\w$-maximal subforest of $G$ (with respect to the tiebreaker $<$).
\end{proposition}
A crucial strengthening of \cref{intro:heavy_splits_into_heavy} is proven in \cref{thm:mf-hyp} and \cref{thm:increasing_seq_in_the_cut}.

Taking a uniformly random linear ordering (tiebreaker) on the edges of $G$, the minimal subforest algorithm yields the \defn{Free Minimal Spanning Forest} (\defn{FMSF}) of $G$ -- a well-known random subforest that has been useful in percolation theory of unimodular graphs.
Analogously, taking a uniformly random linear ordering on the edges of $G$, our $\w$-maximal subforest algorithm yields a random subforest of $G$, which we call the \textbf{Free $\w$-Maximal Spanning Forest} and denote by $\RandomMF_\w(G)$.

FMSF is often applied to random subgraphs $\omega$ of a connected locally finite graph $G$, in particular, to a configuration $\omega$ sampled from an invariant (under automorphisms of $G$) bond percolation $\mathbf{P}$ on $G$, yielding an invariant random subforest $\FMSF(\omega)$ of $G$. 
If the relative weight function $\w$ on $V(G)$ is invariant under a closed subgroup $\Gamma$ of the automorphism group $\Aut(G)$, then the same is true for $\RandomMF_\w$: applied to a sample $\omega$ from a $\Gamma$-invariant percolation $\mathbf{P}$ on $G$, it yields a $\Gamma$-invariant random subforest $\RandomMF_\w(\omega)$ of $G$. 
In fact, such an invariant weight function $\w:=\w_\Gamma$ is induced by $\Gamma$ itself by setting $\w_\Gamma^y(x)$ to be the ratio of the $\Gamma$-Haar measures of the $\Gamma$-stabilizers of vertices $x$ and $y$ of $G$. (See \cref{sec:nonunimod} for details.)

A graph $G$ is called \defn{unimodular} if its automorphism group $\Aut(G)$ is unimodular; equivalently, the relative weight function $\w_\Gamma$ induced by $\Gamma := \Aut(G)$ is constant $1$. 
While nonunimodularity makes some questions easier to answer \cite{Hutchcroft20}, many of the techniques developed for unimodular graphs do not extend to the nonunimodular setting, leaving large gaps in the understanding of percolation processes on such graphs. 
This includes the techniques that make use of $\FMSF(\omega)$ because their correct analogues in this setting would need to involve the induced relative weight function $\w_\Gamma$, which $\FMSF(\omega)$ does not account for, while $\RandomMF_\w(\omega)$ does. 
Even when $G$ is unimodular, it is possible to have a closed nonunimodular subgroup $\Gamma \le \Aut(G)$. For instance, such a subgroup is $\Gamma_\xi \le \Aut(T_d)$ of automorphisms that fix a specified end $\xi$ of the $d$-regular tree $T_d$ with $d \ge 3$. For more examples see \cite{Haggstrom99,Timar06nonu,Hutchcroft20}.

We say that a vertex is a $\w$-\textbf{trifurcation} for $G$ (respectively, cluster $C$) if its deletion splits $G$ (respectively, $C$) into at least three $\w$-nonvanishing connected components. 

\begin{theorem}\label{intro:MF_in_deletion_percolation}
Let $G$ be a locally finite connected graph, $\Gamma$ be a closed subgroup of $\Aut(G)$ whose action on $G$ is transitive, and $\w_\Gamma$ be the $\Gamma$-invariant relative weight function on $V(G)$ induced by $\Gamma$ as above. 
Let $\mathbf{P}$ be a $\Gamma$-invariant percolation on $G$.

Then for $\mathbf{P}$-a.e.\ configuration $\omega$, for every $\w_\Gamma$-heavy cluster $C \subseteq\omega$ which contains a $\w_\Gamma$-trifurcation vertex, a.s.\ the random forest $\RandomMF_{\w_\Gamma}(\omega)$ has a tree $T \subseteq C$ whose space of $\w_\Gamma$-nonvanishing ends is nonempty and perfect\cref{ftnote:perfect_space}.

Moreover, if $\mathbf{P}$ is insertion and deletion tolerant and is such that a.e.\ configuration contains a cluster with $\ge3$ $\w_\Gamma$-nonvanishing ends, then the conclusion holds for every $\w_\Gamma$-heavy cluster.
\end{theorem}
We prove this statement in \cref{thm:MFpercA,thm:MFperc:heavy}.

To bridge the settings of measured graphs and percolation,
in \cref{subsubsec:construction_cluster_graphing}, we give a detailed elementary exposition of the cluster graphing construction of Gaboriau \cite[Sections 2.2--2.3]{Gaboriau05}.
Our version covers the case of nonunimodular transitive graphs and more generally, invariant random structures on a countable set.
It is akin to the construction given in \cite[Section~3.2]{JKLcber}, but includes the treatment of measure and the explicit calculation of the Radon--Nikodym cocycle in \cref{Radon-Nikodym_of_E_V}.

In \cref{subsubsec:applications_cluster_graphings}, we present applications of our results to invariant percolations on transitive graphs (\cref{cluster-graphing_nonamenable,ex:GP}).
We also derive a further combinatorial property called \defn{infinite visibility} (introduced in \cite[8.1]{Ts:hyperfinite_ergodic_subgraph}) of percolation configurations in such graphs (\cref{thm:visibility}).
Lastly, in \cref{thm:indist_forest} we obtain an invariant random subforest of a transitive graph $G$, which has indistinguishable components and is very similar to $\RandomMF_{\w_\Gamma}(G)$, motivating \cref{q:indist} below.

\subsection{Future directions}

In the present work we consider only locally finite Borel graphs. In fact, our analysis heavily relies on compactness of the space of ends for each component of the graph. Noticing that \cref{Gaboriau-Ghys,Anush-Robin} do not have this restriction, it is of interest to generalize \cref{intro:witnessing_nonamenability}.

\begin{question}\label{q:locally_ctbl}
Does \cref{intro:witnessing_nonamenability} extend to the locally infinite setting?
\end{question}

Both \cref{intro:witnessing_nonamenability,intro:MF_in_deletion_percolation} yield a forest that contains a tree, whose space of nonvanishing ends is nonempty and perfect, but the theorems do not claim that it is a closed subset of the space of all ends. It is of interest to further understand this space.

We think that actually all nonvanishing ends should always be of the same weight, i.e., in case of a unique nonvanishing end of infinite weight, all other ends of that component should be vanishing (for a.e.\ component).
In fact, this would follow from a positive answer to the following more general question, which has a positive answer for pmp graphs due to Epstein and Hjorth \cite[Theorem 1.4]{Epstein-Hjorth:end-selection}:

\begin{question}\label{q:selected_end}
Let $G$ be a quasi-pmp Borel graph and let $\w$ denote its Radon--Nikodym cocycle with respect to the underlying probability measure. 
If there is a Borel selection of one $\w$-nonvanishing end in a.e.\ $G$-component, then is it true that a.e.\ $G$-component admits at most 2 nonvanishing ends?
\end{question}

\begin{addendum*}
\Cref{q:locally_ctbl,q:selected_end} have been positively resolved by the authors and R.~Tucker-Drob in a forthcoming work \cite{CTTTD}.
\end{addendum*}

On the percolation side, recall that for a transitive connected locally finite graph $G$, $\RandomMF_{\w_\Gamma}(G)$ is a natural generalization of $\FMSF(G)$ to the nonunimodular setting.
Hence it would be interesting to extend various properties of $\FMSF(G)$ to its more general counterpart.
\cref{intro:MF_in_deletion_percolation} gives at least one tree whose set of nonvanishing ends is nonempty and perfect. 
Tim\'ar \cite{timar2006ends} showed that for a transitive unimodular graph $G$, the number of ends of every tree in $\FMSF(G)$ is the same.

\begin{question}\label{q:everytree}
For a transitive nonunimodular graph $G$, is the number of nonvanishing ends the same for every tree of $\RandomMF_{\w_\Gamma}(G)$?
\end{question}

More generally, one can ask:

\begin{question}\label{q:indist}
Are $\w_\Gamma$-heavy trees of $\RandomMF_{\w_\Gamma}(G)$ indistinguishable in the sense of \cite{LyonsSchramm99}?
\end{question}

When $G$ is unimodular and $\FMSF(G)$ has infinitely many ends, the indistinguishability of $\FMSF(G)$ was shown in \cite[Theorem 1.2]{Timar18ind}.
One consequence of this result is a simplification of the construction of the treeable ergodic subrelation in the Gaboriau--Lyons theorem \cite[Proposition 13]{Gaboriau-Lyons}. 
The indistinguishable ``cousin'' of the Free ${\w_\Gamma}$-Maximal Spanning Forest on $G$ that we construct in \cref{thm:indist_forest} serves as a further motivation for the last question.



\subsection*{Organization}

The rest of the paper is organized as follows. In \cref{sec:prelim} we discuss preliminaries. 
In \cref{sec:MF} we present the construction of the maximal forest in the Borel setting and prove \cref{intro:witnessing_nonamenability}. \cref{sec:MFperc} reviews the significance of random spanning forests in percolation theory and presents our construction in the context of this theory. Finally, in \cref{sec:appl} we present the cluster graphing construction and applications of our main results on concrete examples.

\subsection*{Acknowledgments}
We thank Russell Lyons, \'Ad\'am Tim\'ar, and Robin Tucker-Drob for many insightful conversations, as well as the anonymous referees for helpful comments and corrections.
R.C.\ was supported by NSF Grant DMS-2224709.
G.T.\ and A.Ts.\ were supported by NSF Grant DMS-1855648, and A.Ts.\ was also supported by NSERC Discovery Grant RGPIN-2020-07120.

\section{Preliminaries}
\label{sec:prelim}

\subsection{Graphs}
\label{sec:prelim-graph}

Throughout this paper, by a \defn{graph} we mean a simple undirected graph, represented formally as a symmetric \emph{reflexive} binary edge relation $G \subseteq X^2$ on the vertex set $X$.
We therefore write
\begin{equation*}
(x,y) \in G  \iff  x \mathrel{G} y
\end{equation*}
interchangeably to mean that there is an edge from $x$ to $y$.
We will refer to the graph by $(X,G)$ or, when $X$ is clear from context, simply by $G$.

\begin{definition}
\label{def:graph-conn}
By a \defn{connected} graph $(X,G)$, we will mean one \emph{whose vertex set $X$ is nonempty} and such that any $x, y \in X$ are joined by a path.
A subset $A \subseteq X$ is \defn{$G$-connected} if the induced subgraph $G|A$ is.

We write $\#E_G \subseteq X^2$ for the \defn{induced equivalence relation} relating two vertices iff they are joined by a path, and write $X/G := X/\#E_G$ for the quotient set, i.e., connected components of $G$.
Note that since we are considering reflexive graphs, $\#E_G$ is itself a graph.
For a subset $A \subseteq X$, we write $[A]_G := [A]_{\#E_G}$ for its \defn{$G$-saturation}, i.e., the union of all components intersecting $A$.
\end{definition}

\begin{definition}
A \defn{simple path} $x_0 \mathrel{G} x_1 \mathrel{G} \dotsb \mathrel{G} x_n$ will mean one with no repeated vertices.
A \defn{simple cycle} will mean such a path with $n \ge 3$, $x_0 = x_n$ and no other repeated vertices.
(Recall that we are working with reflexive graphs.)
\defn{Acyclic} means there are no simple cycles.

A \defn{forest} is an acyclic graph; a \defn{tree} is a connected forest.
\end{definition}

\begin{definition}
\label{def:graph-bdy}
Given a graph $(X,G)$ and subset $A \subseteq X$,
\begin{itemize}
\item
the \defn{inner (vertex) boundary} of $A$ is the set of vertices in $A$ adjacent to $X \setminus A$;
\item
the \defn{outer (vertex) boundary} of $A$ is the inner boundary of $X \setminus A$;
\item
the \defn{edge boundary} of $A$ is the set of all $G$-edges with one end-vertex in $A$ and the other in $X \setminus A$.
\end{itemize}
Note that if $G$ is locally finite, then one of these notions of boundary of $A$ is finite iff all are, in which case we say that $A$ has \defn{finite boundary} or is \defn{boundary-finite}.
\end{definition}

\begin{remark}
A connected locally finite graph has only countably many boundary-finite subsets.
\end{remark}

\subsection{End spaces}
\label{sec:prelim-ends}

In this work, we find it convenient to take the following point-set topological approach to ends, which is rooted in the origins of the concept due to Freudenthal \cite{Freends} and Hopf \cite{Hopends}, and explicitly introduced in a graph-theoretic context by Polat \cite{Polatends} (phrased in the slightly different but equivalent context of uniformities).

\begin{definition}
\label{def:ends-space}
For a connected locally finite graph $(X,G)$, its \defn{end compactification}
\begin{equation*}
\^X = \^X^G \supseteq X
\end{equation*}
is the space of all ultrafilters on the countable Boolean algebra of boundary-finite subsets of vertices $A \subseteq X$, i.e., the Stone space of said algebra, where we identify vertices $x \in X$ with principal ultrafilters in $\^X$.
In other words, $\^X$ is the unique zero-dimensional Polish compactification of the discrete space $X$ whose clopen sets are precisely the closures of boundary-finite $A \subseteq X$; we denote said closures by $\^A \subseteq \^X$.

The \defn{end space} of $(X,G)$ is
\begin{equation*}
\partial X = \partial^G X := \^X^G \setminus X,
\end{equation*}
or equivalently the closed subspace of nonisolated points in $\^X^G$, which are the \defn{ends} of $G$.
\end{definition}

\begin{remark}[Various definitions of ends]
\label{rmk:ends-otherdef}
There are several equivalent ways to present the definition of ends and end compactification for locally finite graphs.
One way, originating with Halin \cite{Halends} and standard in combinatorics, is to define ends as equivalence classes of rays (meaning infinite simple paths $(x_n)_{n \in \#N}$), where we consider two rays equivalent if their tails stay in the same connected component after the removal of any finite subset of vertices (which is equivalent to the existence of a third ray intersecting each of the two rays infinitely often).
Instead of rays, one could also consider arbitrary sequences of vertices whose tails stay in the same component after removing any finite set; these are the so-called end-convergent sequences, commonly used in percolation theory \cite[page~242]{LyonsBook}.

To see the equivalence between the definition with, say, rays, and \cref{def:ends-space} in terms of ultrafilters, note that an equivalence class of rays is uniquely determined by the ultrafilter of those boundary-finite sets containing one of their tails.
Conversely, it is easily seen that every nonprincipal ultrafilter of boundary-finite sets is determined in this manner by some ray; thus there is a canonical bijection between the ray-based and ultrafilter-based definitions of ends.
However, we find the latter more convenient for our purposes in this paper, since it emphasizes the primacy of neighborhoods of ends, as opposed to limits of sequences which are awkward when working with ends topologically (as in e.g., \cref{def:weights-nonvanish}).
\end{remark}

\begin{remark}
\label{rmk:ends-inf}
A clopen set $\^A \subseteq \^X$ contains an end iff $A \subseteq X$ is infinite, by compactness.
\end{remark}

\begin{remark}
For boundary-finite $A \subseteq X$, it is easily seen that $B \subseteq A$ has finite boundary in $G$ iff it has finite boundary in the induced subgraph $G|A$.
In other words, the notation $\^A$ may also be consistently interpreted as the end compactification of $(A,G|A)$, which embeds into $\^X$ as a clopen subspace.
Because of this, we will sometimes refer to an end $\xi\in \partial^G X$ which is in $\^A$ as an \defn{end in $A$} or say that \defn{$A$ is a neighborhood of $\xi$}.

In contrast, \emph{for boundary-infinite $Y \subseteq X$, we must carefully distinguish between $\^Y^{G|Y}$ and the closure $\-Y$ of $Y$ in $\^X^G$}.
There is always a canonical map from the former space to the latter, but it need not be injective; see \cref{ex:ends-map-noninj} below.
\end{remark}

\begin{remark}
\label{rmk:ends-basis-conn}
Every boundary-finite $A \subseteq X$ is the finite (disjoint) union of its connected components, which are also boundary-finite.
Thus, the closures $\^C \subseteq \^X$ of \emph{connected} boundary-finite $C \subseteq X$ also form an open basis for $\^X$.
\end{remark}

\begin{lemma}
\label{thm:ends-basis}
The following families of subsets of $X$ are the same:
\begin{enumerate}[label=(\roman*)]
\item \label{thm:ends-basis:cut}
connected boundary-finite $C \subseteq X$ such that $X \setminus C$ is also connected;
\item \label{thm:ends-basis:furc}
connected components $C$ of $X \setminus F$ for some finite connected $F \subseteq X$.
\end{enumerate}
The family of $\^C \subseteq \^X$ for all (infinite) such $C$ form a neighborhood basis for each end $\xi \in \partial X$.
\end{lemma}
\begin{proof}
Clearly \cref{thm:ends-basis:furc}$\implies$\cref{thm:ends-basis:cut}; to see the converse, let $F$ be the outer boundary of $C$ together with finitely many paths in $X \setminus C$ to make $F$ connected.
That such $\^C$ form a basis for ends is the trivial $n=1$ case of \cref{thm:ends-furc} below.
\end{proof}

\begin{definition}
\label{def:ends-furc}
For a finite connected $F \subseteq X$, we call the components of $X \setminus F$ the \defn{sides} of $F$.

For $n \in \#N^+$, an \defn{$n$-furcation} is a finite (nonempty) connected set $F \subseteq X$ with at least $n$ infinite sides.
An \defn{$n$-furcation vertex} is a vertex $x$ such that $\{x\}$ is an $n$-furcation.
When $n = 2, 3$, we say \defn{bifurcation}, \defn{trifurcation} respectively.
(Note that a trifurcation is also a bifurcation.)
\end{definition}

\begin{lemma}
\label{thm:ends-furc}
$\abs{\partial X} \ge n$ iff there is at least one $n$-furcation.
In that case, for any $n$ distinct ends $u_1, \dotsc, u_n \in \partial X$ and clopen neighborhoods $u_i \in \^A_i \subseteq \^X$, there is an $n$-furcation $F$ and distinct sides $C_i \subseteq X \setminus F$ of it such that $u_i \in \^C_i \subseteq \^A_i$.
(In other words, the products of distinct sides of $n$-furcations form a neighborhood basis for each pairwise distinct $(u_1, \dotsc, u_n) \in (\partial X)^n$.)
\end{lemma}
\begin{proof}
If there is an $n$-furcation, then (the closures of) its $\ge n$ infinite sides each contain an end (by \cref{rmk:ends-inf}).
Conversely, if $u_1, \dotsc, u_n$ are distinct ends, each contained in a clopen neighborhood $\^A_i$, then we may find $u_i \in \^B_i \subseteq \^A_i$ such that the $\^B_i$ are pairwise disjoint, and let $F$ be the union of the outer boundaries of the $B_i$ together with finitely many paths to make $F$ connected; then the $u_i$ must belong to (the closures of) distinct sides $C_i$ of $F$, whence $F$ is an $n$-furcation.
\end{proof}

\begin{definition}
\label{def:ends-map}
Let $f : X -> Y$ be a map between the vertex sets of two connected locally finite graphs $(X,G)$ and $(Y,H)$.
We extend $f$ by continuity to a (partial) map
\begin{align*}
\^f : \^X^G &--> \^Y^H \\
\xi &|--> \lim_{X \ni x -> \xi} f(x),
\end{align*}
where this limit exists; it clearly exists and equals $f(x)$ for vertices $x \in X$.
If it also exists for every end $\xi \in \partial^G X$, we call $\^f$ the \defn{map induced by $f$} (it is then automatically continuous). When $(X,G)$ is a subgraph of $(Y,H)$, we denote the inclusion map by $\iota : X -> Y$ and the induced map by $\^\iota : \^X^G -> \^Y^H$.
\end{definition}


\begin{lemma}
\label{thm:ends-map-fin1hom}
If $f : (X,G) -> (Y,H)$ is a finite-to-one graph homomorphism, in particular if $f$ is the inclusion of a subgraph, then the induced map $\^f$ exists, and restricts to a map $\partial^G X -> \partial^H Y$.
\end{lemma}
\begin{proof}
As $X \ni x -> \xi\in \partial X$, $f(x)$ cannot cluster around a vertex $y \in Y$, since $A := f^{-1}(y) \subseteq X$ is finite and so $\^{X \setminus A}$ is a neighborhood of $\xi$ which $f$ maps to $Y \setminus \{y\}$.
It remains to rule out the possibility that $f(x)$ clusters around two distinct ends $\zeta_1, \zeta_2 \in \partial Y$.
Indeed, let $\^A \subseteq \^Y$ be a clopen set such that $\zeta_1 \in \^A \not\ni \zeta_2$, with $A \subseteq Y$ boundary-finite.
Since $f$ is a graph homomorphism, $f$ maps the inner boundary of $f^{-1}(A)$ into that of $A$; since $f$ is also finite-to-one, $f^{-1}(A)$ thus has finite boundary, and so either $\^{f^{-1}(A)} \subseteq \^X$ or its complement is a neighborhood of $\xi$, but not both, which means $f(x)$ cannot cluster around both $\zeta_1 \in \^A$ and $\zeta_2 \in \^Y \setminus \^A$ as $x -> \xi$.
\end{proof}

\begin{example}
\label{ex:ends-rays}
If $R$ is the infinite ray graph $0 -- 1 -- 2 -- \dotsb$ on vertices $\#N$, then an injective graph homomorphism $f : (\#N,R) -> (X,G)$ takes the unique end of $R$ to an end of $G$.
This recovers the correspondence with the ray-based definition of ends, as in \cref{rmk:ends-otherdef}.
\end{example}

\begin{example}
\label{ex:ends-map-noninj}
Even for the inclusion $\iota : (X,G) -> (Y,H)$ of a subgraph, with either the same vertex set $X = Y$ and a subset of edges $G \subseteq H$, or the induced subgraph $G = H|X$ on a subset of vertices $X \subseteq Y$, there is no reason for the induced map $\^\iota : \partial^G X -> \partial^H Y$ to be injective.
The square lattice graph on $Y = \#Z^2$ is one-ended; by removing either vertices or edges, we can turn it into a tree with $2^{\aleph_0}$ ends.
\end{example}

\begin{lemma}
\label{thm:ends-map-conn}
Under the assumptions of \cref{thm:ends-map-fin1hom}, if also $f^{-1}$ preserves (nonempty) connected subsets (it is enough to check 1- and 2-element subsets), then $\^f$ restricts to a homeomorphism $\partial^G X \cong \partial^H Y$.
\end{lemma}
\begin{proof}
Recall from \cref{def:graph-conn} that ``connected'' includes ``nonempty''; thus $f$ is surjective, hence so is $\^f$ by the density of $Y \subseteq \^Y^H$.
To check injectivity: let $\xi, \zeta \in \partial X$ such that $\^f(\xi) = \^f(\zeta)$.
Then for any finite connected $F \subseteq X$, since $f$ is a graph homomorphism, $f(F) \subseteq Y$ is still connected (and finite); and $\^f(\xi) = \^f(\zeta)$ lies on one side $D \subseteq Y \setminus f(F)$ of it.
Since $f^{-1}$ preserves connectedness, $f^{-1}(D) \subseteq X \setminus f^{-1}(f(F)) \subseteq X \setminus F$ is contained in one side of $F$, which thus contains both $\xi, \zeta$.
So $\xi, \zeta$ lie on the same side of every finite connected $F$, whence $\xi=\zeta$ (by \cref{thm:ends-basis}).
\end{proof}

\begin{definition}
\label{def:ends-nonconn}
For a possibly disconnected locally finite graph $(X,G)$, we define its \defn{end compactification}, respectively \defn{end space}, to be the disjoint union of those of its components:
\begin{align*}
\^X = \^X^G &:= \bigsqcup_{C \in X/G} \^C^G, \\
\partial X = \partial^G X &:= \bigsqcup_{C \in X/G} \partial^G C = \^X^G \setminus X.
\end{align*}
Note that these are locally compact spaces.
The notions of \defn{$n$-furcation} and \defn{side} are interpreted as in \cref{def:ends-furc} within a single $G$-component.

For a map $f : X -> Y$ between the vertex sets of two such graphs $(X,G), (Y,H)$, we define the induced map $\^f : \^X^G -> \^Y^H$ componentwise (i.e., on $\^C^G$ for each $C \in X/G$) as in \cref{def:ends-map}.
This map is guaranteed to exist everywhere if the conditions of \cref{thm:ends-map-fin1hom} are satisfied componentwise, i.e., $f|C : C -> Y$ is a finite-to-one graph homomorphism for each $C \in X/G$.
\end{definition}

\subsection{Weighted graphs and ends}
\label{sec:prelim-weights}

We denote by $\#R^+$ the multiplicative group of positive reals.

\begin{definition}\label{def:w}
A \defn{weight function} on a graph $(X,G)$ is an arbitrary function $\w : X -> \#R^+$. We often treat $\w$ as an atomic measure on $X$, writing $\w(A) := \sum_{x \in A} \w(x)$ for a set $A \subseteq X$.
\end{definition}

We call a set $A \subseteq X$ \defn{$\w$-finite} if $\w(A) < \infty$; otherwise, we call it \defn{$\w$-infinite}. These notions are respectively called \textbf{$\w$-light} and \textbf{$\w$-heavy} in percolation theory.

\begin{definition}\label{def:weights_of_ends}
\label{def:weights-nonvanish}
Let $\w : X -> \#R^+$ be a weight function.

For an arbitrary subset $A \subseteq X$, we put
\begin{equation*}
\limsup_A \w = \limsup_{x \in A} \w(x) := \inf_{\text{finite } F \subseteq A} \sup_{x \in A \setminus F} \w(x) \in [0,\infty].
\end{equation*}
If this quantity is $0$, we say $A$ is \defn{($\w$-)vanishing}; otherwise $A$ is \defn{($\w$-)nonvanishing}. If $\limsup_A \w <\infty$, we say that $A$ is \defn{($\w$-)bounded}, otherwise $A$ is \defn{($\w$-)unbounded}.

For an end $\xi\in \partial X$, we put
\begin{equation*}
\^\w(\xi) := \limsup_{x -> \xi} \w(x) = \inf_{\^A \ni \xi} \sup_{x \in A} \w(x) \in [0,\infty]
\end{equation*}
(where $\^A$ ranges over clopen neighborhoods of $\xi$).
In other words, $\^\w : \^X -> [0,\infty]$ is the minimal upper semicontinuous extension of $\w$.
If $\^\w(\xi) = 0$, we say that the end $\xi$ is \defn{($\w$\nobreakdash-)vanishing}; otherwise $\xi$ is \defn{($\w$-)nonvanishing}. Similarly, if $\^\w(\xi) < \infty$ we say that $\xi$ is \defn{($\w$-)bounded}, otherwise $\xi$ is \defn{($\w$-)unbounded}.

Let
\begin{equation*}
\partial_\w X = \partial^G_\w X := \set{\xi\in \partial X}{\text{$\xi$ is $\w$-nonvanishing}}.
\end{equation*}
\end{definition}

\begin{remark}
\label{rmk:ends-nonvanish-fsigma}
By upper semicontinuity, $\partial_\w X = \bigcup_n \^\w^{-1}([1/n,\infty]) \subseteq \partial X$ is an $F_\sigma$ subset.
It may not be $G_\delta$, as in \cref{fig:ends-nonvanish-fsigma} where it is a countable dense set.
In other words, \emph{$\partial_\w X \subseteq \partial X$ with the subspace topology may not be Polish!}
\end{remark}

\begin{figure}[H]
\centering
\includegraphics{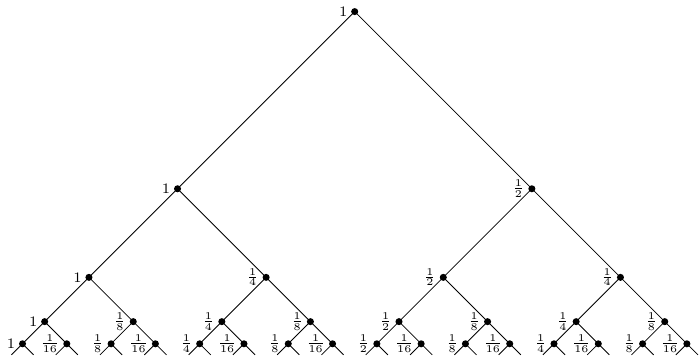}
\caption{A weighted tree with set of nonvanishing ends $F_\sigma$ but not $G_\delta$}
\label{fig:ends-nonvanish-fsigma}
\end{figure}

The notions of vanishing for sets and ends are related as follows (this is analogous to \cref{rmk:ends-inf}):

\begin{lemma}
\label{thm:weights-nonvanish}
\leavevmode
\begin{enumerate}[label=(\alph*)]
\item \label{thm:weights-nonvanish:end}
For an end $\xi\in \partial X$,
\begin{equation*}
\^\w(\xi) = \inf_{\^A \ni \xi} \limsup_A \w.
\end{equation*}
Thus if $\xi$ is nonvanishing, then every boundary-finite $A \subseteq X$ containing $\xi$ is nonvanishing.
\item \label{thm:weights-nonvanish:set}
For an infinite boundary-finite $A \subseteq X$ contained in a single $G$-component,
\begin{equation*}
\limsup_A \w = \max_{\xi \in \^A} \^\w(\xi).
\end{equation*}
Thus $A$ is nonvanishing iff it contains a nonvanishing end.
\end{enumerate}
\end{lemma}
\begin{proof}
\cref{thm:weights-nonvanish:end}
Clearly $\^\w(\xi) = \inf_{\^A \ni \xi} \sup_A \w \ge \inf_{\^A \ni \xi} \limsup_A \w$;
conversely, for each $\^A \ni \xi$, we have $\^\w(\xi) = \inf_{\^B \ni \xi} \sup_B \w \le \limsup_A \w$, since $A$ minus any finite set contains a neighborhood $B$ of $\xi$.

\cref{thm:weights-nonvanish:set}
First, note that by the compactness of $\^A$ (because $A$ is contained in a single $G$-component) and upper semicontinuity of $\^\w$, the maximum is achieved.
Now $\ge$ follows from \cref{thm:weights-nonvanish:end}.
To show $\le$: if $\limsup_A \w > 0$, then there is a sequence of distinct vertices $x_0, x_1, \dotsc \in A$ such that $\lim_{n -> \infty} \w(x_n) = \limsup_A \w$; a subsequence of these converges to some end $\xi \in \^A$ with $\^\w(\xi) \ge \limsup_A \w$.
\end{proof}

\begin{remark}
The converse of the last statement of \cref{thm:weights-nonvanish}\cref{thm:weights-nonvanish:end} is false, as shown by the example in \cref{fig:ends-nonvanish-fsigma}.
In fact, \cref{thm:weights-nonvanish}\cref{thm:weights-nonvanish:set} shows that every neighborhood of an end $\xi$ is nonvanishing iff $\xi$ belongs to the \emph{closure} $\-{\partial_\w X} \subseteq \partial X$ of the set of nonvanishing ends.
\end{remark}

\begin{definition}[cf.\ \cref{def:ends-furc}]
\label{def:weights-furc}
A \defn{$\w$-$n$-furcation} is a finite connected $F \subseteq X$ with at least $n$ nonvanishing sides $C_1, \dotsc, C_n \subseteq [F]_G \setminus F$.
A \defn{$\w$-$n$-furcation vertex} is a singleton $\w$-$n$-furcation.
When $n = 2, 3$, we say \defn{$\w$-bifurcation}, \defn{$\w$-trifurcation} respectively.
\end{definition}

\begin{lemma}[cf.\ \cref{thm:ends-furc}]
\label{thm:weights-furc}
For connected $G$, $\abs{\partial_\w X} \ge n$ iff there is at least one $\w$-$n$-furcation.
In that case, for any $n$ distinct $\w$-nonvanishing ends $u_1, \dotsc, u_n \in \partial_\w X$ and clopen neighborhoods $u_i \in \^A_i \subseteq \^X$, there is a $\w$-$n$-furcation $F$ and distinct sides $C_i \subseteq X \setminus F$ of it such that $u_i \in \^C_i \subseteq \^A_i$.
(In other words, the products of distinct sides of $\w$-$n$-furcations form a neighborhood basis for each pairwise distinct $(u_1, \dotsc, u_n) \in (\partial_\w X)^n$.)
\end{lemma}
\begin{proof}
Analogous to \cref{thm:ends-furc}, using \cref{thm:weights-nonvanish} in place of \cref{rmk:ends-inf}.
\end{proof}

The following results record the behavior of weights of ends upon passing to a subgraph or quotient graph (see e.g., \cref{def:quot-fin}):

\begin{lemma}
\label{thm:weights-map}
Let a finite-to-one homomorphism between connected locally finite graphs $f : (X,G) -> (Y,H)$ induce $\^f : \^X^G -> \^Y^H$ as in \cref{thm:ends-map-fin1hom}.
For a weight function $\w : X -> \#R^+$, put
\begin{align*}
\sup\nolimits_f \w : Y &--> \#R^+ \\
y &|--> \sup_{x \in f^{-1}(y)} \w(x).
\end{align*}
Then for an end $\zeta \in \partial^H Y$,
\begin{equation*}
\tag{$*$}
\^{\sup\nolimits_f \w}(\zeta) = \sup_{\xi \in \^f^{-1}(\zeta)} \^\w(\xi).
\end{equation*}
Thus for $\xi \in \partial^G X$,
\begin{equation*}
\^\w(\xi) \le \^{\sup\nolimits_f \w}(\^f(\xi)).
\end{equation*}
In particular, $\^f$ restricts to a map $\partial^G_\w X -> \partial^H_{\sup_f \w} Y$ between spaces of nonvanishing ends.
\end{lemma}
\begin{proof}
Both sides of ($*$) define the least upper semicontinuous map $\^Y^H -> [0,\infty]$ whose composite with $f : X -> Y \subseteq \^Y^H$ is $\ge \w$.
\end{proof}

\begin{corollary}
\label{thm:weights-map-nonconn}
Let a componentwise finite-to-one homomorphism between (possibly disconnected) locally finite graphs $f : (X,G) -> (Y,H)$ induce $\^f : \^X^G -> \^Y^H$ as in \cref{def:ends-nonconn}.
For weight functions $\w_G : X -> \#R^+$ and $\w_H : Y -> \#R^+$ such that
$\w_G \le \w_H \circ f$,
we have for each $\xi\in \partial^G X$,
\begin{equation*}
\^{\w_G}(\xi) \le \^{\w_H}(\^f(\xi)).
\end{equation*}
In particular, $\^f$ restricts to a map $\partial^G_{\w_G} X -> \partial^H_{\w_H} Y$ between spaces of nonvanishing ends.
\end{corollary}
\begin{proof}
For each component $C \in X/G$ mapping into $D := [f(C)]_H \in Y/H$, since
$\w_G \le \w_H \circ f$,
we have
$\sup_{f|C}(\w_G|C) \le \w_H|D$,
so we may apply \cref{thm:weights-map} to $f|C : C -> D$.
\end{proof}

In light of \cref{thm:ends-map-fin1hom,thm:weights-map-nonconn}, we use the following notion:

\begin{definition}\label{def:canonical_end-map}
Let $(Y,H)$ be a graph equipped with a weight function $\w$. For a subgraph $(X,G)$ of $(Y,H)$, the \defn{canonical maps} $\partial^G X -> \partial^H Y$ and $\partial^G_\w X -> \partial^H_\w Y$ are the restrictions of the map $\^\iota : \^X^G -> \^Y^H$ induced by the inclusion $\iota : X \to Y$. We also refer to the $\iota$-images of ends of $(X,G)$ as \textbf{canonical images}.
\end{definition}

\begin{remark}
\label{rmk:weights-homog}
It will be important below that all of the notions considered in this subsection are \defn{homogeneous in $\w$}, meaning preserved under scaling $\w$ by any constant in $\#R^+$.
\end{remark}

\subsection{Relative weight functions (cocycles) on graphs}\label{sec:cocycle}

In the sequel, we use the notion of $\w$-vanishing sets and ends for graphs equipped with a \emph{relative} weight function $\w$, which we now define. Let $(X,G)$ be a locally finite possibly disconnected graph.

\begin{definition}
\label{def:relative_weight}
An \defn{$\#R^+$-valued cocycle} or a \textbf{relative weight function} on $G$ is a map $\w : G -> \#R^+$ satisfying the \defn{cocycle identity}
\begin{equation*}
\w(x_0,x_1) \w(x_1,x_2) \dotsm \w(x_{n-1},x_n) = 1 \quad \text{for any cycle $x_0 \mathrel{G} x_1 \mathrel{G} \dotsb \mathrel{G} x_n = x_0$}.
\end{equation*}
Such $\w$ then extends uniquely to a cocycle on the induced equivalence relation $\#E_G$, which we also denote $\w$, namely $\w(x,y) := \w(x_0,x_1) \dotsm \w(x_{n-1},x_n)$ for any path $x = x_0 \mathrel{G} x_1 \mathrel{G} \dotsb \mathrel{G} x_n = y$.
\end{definition}

For vertices $x,y$ in the same component of $G$, we think of $\w^y(x) := \w(x,y)$ as the weight of $x$ relative to $y$. Indeed, the map 
\begin{equation*}
\w^y := \w(-,y) : [y]_G --> \#R^+
\end{equation*}
is simply a weight function on the $G$-component of $y$, and these weight functions $\w^y$ and $\w^z$ for different basepoints $y,z$ in the same $G$-component are constant multiples of each other by the cocycle identity $\w^z = \w^z(y) \w^y$. Because of this, for a fixed $G$-component $C$, \defn{homogeneous} statements about $\w^b$ do not depend on the choice of the basepoint $b \in C$; for example:
\begin{itemize}
\item the definitions of $\w^b$-finite, $\w^b$-vanishing, $\w^b$-$n$-furcation for sets and ends in $C$,

\item $\w^b(x) < \w^b(y)$ for $x,y \in C$,

\item $\min \{\w^b(x), \w^b(y)\} \le \min \{\w^b(u), \w^b(v)\}$ for $x,y,u,v \in C$.
\end{itemize}
We drop $b$ from the superscript in such ($\w$-homogeneous) statements and simply write $\w$, e.g.\ $\w$-nonvanishing. In particular, per \cref{rmk:weights-homog}, we may use the notions and statements of \cref{sec:prelim-weights} for a \emph{relative} weight function $\w$ on $G$.

\subsection{Borel and quasi-pmp graphs and equivalence relations}\label{sec: CBER and quasi-pmp graphs}

Let $(X,G)$ be a locally finite Borel graph, i.e., the vertex set $X$ is a standard Borel space, and $G \subseteq X^2$ is Borel as a set of pairs.

\begin{remark}
\label{rmk:borel-nonconn}
In general, notions of \defn{end space}, etc., for $(X,G)$ are to be understood in the general sense of disconnected locally finite graphs, as in \cref{def:ends-nonconn}.
Thus for example, $\^X^G$ is typically a nonseparable locally compact Hausdorff space.
Note that the topology on $\^X^G$ has nothing to do with any compatible Polish topology on $X$.
\end{remark}

\begin{definition}
\label{def:qpmp}
Let $\mu$ be a probability measure on $X$.

We say that $\mu$ is \defn{($G$-)quasi-invariant} (or that $G$ is a \defn{quasi-pmp} graph) if for every Borel $\mu$-null $A \subseteq X$, $[A]_G$ is still $\mu$-null.

For a Borel cocycle $\w : G -> \#R^+$, we say that $\mu$ is \defn{$\w$-invariant} if for any Borel sets $A, B \subseteq X$ and Borel bijection $\gamma : A \cong B$ with graph contained in $G$ (i.e., perfect $G$-matching between $A,B$),
\begin{equation*}
\mu(B) = \int_A \w^x(\gamma(x)) \,d\mu(x).
\end{equation*}
It follows that the same holds for $\gamma$ with graph contained merely in $\#E_G$.

In fact, it is enough to require this equation only for countably many Borel bijections $\gamma$ whose graphs cover $G$.
For instance, if $G$ is the Schreier graph of a Borel action of a countable group $\Gamma \curvearrowright X$, then it is enough to require this for $\gamma$ among the generators of $\Gamma$.
See \cite[8.1, 2.1]{KMtopics}.

If $\mu$ is $\w$-invariant, then it is clearly quasi-invariant.
Conversely, every quasi-invariant $\mu$ is $\w$-invariant for an essentially unique (mod $\mu$-null) Borel cocycle $\w : G -> \#R^+$, called the \defn{Radon--Nikodym cocycle} of $\#E_G$ with respect to $\mu$; see \cite[8.3]{KMtopics}.
\end{definition}

A \defn{countable Borel equivalence relation}  (CBER) $E \subseteq X^2$ is a Borel equivalence relation with countable classes; see \cite{Kcber} for general background.
These are exactly the connectedness relations $\#E_G$ of locally finite Borel graphs $G$ \cite[remark after proof of 3.12]{JKLcber}.

A CBER $E$ on $X$ is called
\begin{itemize}
\item \defn{smooth} if it has a \defn{Borel transversal} $A \subseteq X$, meaning a Borel set containing exactly one element from each $E$-class.

\item \defn{hyperfinite} if it is an increasing union of finite Borel equivalence relations.

\item \defn{amenable} if there is a sequence of Borel functions $\lambda_n: E\to[0,1]$ that are summable to $1$ on each equivalence class and for all $(x,y)\in E$ we have that $\norm{\lambda_n(x,\cdot)-\lambda_n(y,\cdot)}_1\to 0$ as $n$ tends to infinity.

\item \defn{treeable} if it admits an acyclic graphing, where a \textbf{graphing} of $E$ is a Borel graph $G$ on $X$ whose connectedness relation $\#E_G$ is $E$.
\end{itemize}
In the presence of a Borel probability measure $\mu$ on $X$, the notions of hyperfinite, amenable, and treeable are relaxed to $\mu$-hyperfinite, $\mu$-amenable, and $\mu$-treeable by demanding that the corresponding property holds off of a $\mu$-null set.
We also use the notions of $\mu$-amenability and $\mu$-hyperfiniteness interchangeably because they are equivalent by the Connes--Feldman--Weiss theorem \cite{Connes-Feldman-Weiss}. 
Finally, we often omit $\mu$ before these terms when it is clear from the context.

\begin{lemma}
\label{thm:cber-compress}
If $(X,E)$ is a smooth countable Borel equivalence relation, $\w : E -> \#R^+$ is a Borel cocycle, and each $E$-class is $\w$-infinite, then there are no $\w$-invariant probability measures.
\end{lemma}
This fact is well-known; see \cite[2.1]{Mnad}, \cite[5.6]{Ts:hyperfinite_ergodic_subgraph}.
For the reader's convenience, we include the easy proof.
\begin{proof}
Let $A \subseteq X$ be a Borel transversal.
By Lusin--Novikov uniformization \cite[18.10]{Kcdst}, there are Borel maps $\gamma_0, \gamma_1, \dotsc : A -> X$ such that for each $x \in A$, $(\gamma_i(x))_i$ is an injective enumeration of $[x]_E$.
Then for any $\w$-invariant $\mu$,
\begin{align*}
\mu(X)
&= \sum_i \mu(\gamma_i(A))
= \sum_i \int_A \w^x(\gamma_i(x)) \,d\mu(x)
= \int_A \sum_{y \in [x]_E} \w^x(y) \,d\mu(x)
= \int_A \infty \,d\mu(x).
\qedhere
\end{align*}
\end{proof}

\section{$\w$-maximal subforests}\label{sec:MF}

In this section, we present our main cycle-cutting algorithm mentioned in \cref{cycle-cutting_algorithm}. 
We do so in several stages, starting with a connected locally finite graph equipped with a relative weight function and building our way up to locally finite quasi-pmp graphs.

\subsection{For a connected graph with enough trifurcation vertices}\label{subsec:with_trifurc-vertices}

Throughout this subsection, we let $(X,G)$ be a connected locally finite graph with a relative weight function $\w : G -> \#R^+$. 
Fixing a basepoint $b \in X$, we get a genuine weight function $\w^b$ on $X$, which we use below, omitting the superscript $b$ from $\w$-homogeneous statements as they do not depend on the choice of the basepoint $b$. 
Especially in this subsection, the reader can think of $\w$ as a weight function on $X$ without any harm.

\begin{definition}
\label{def:w-order}
We extend the weight function $\w^b$ from $X$ to (the edge-set of) $G$ by setting 
\[
\~\w^b(e) := \min\{\w^b(x),\w^b(y)\}
\]
for an edge $e=\{x,y\} \in G$. 
Fix also an arbitrary linear ordering $\ledge$ on the \emph{undirected} $G$-edges.
Define a new linear ordering on the undirected $G$-edges as follows: for $e_1,e_2 \in G$,
\begin{equation*}
e_1 \lw e_2
\coloniff
\~\w(e_1) < \~\w(e_2)
\OR \big[\~\w(e_1) = \~\w(e_2) \AND e_1 \ledge e_2\big].
\end{equation*}
We emphasize that the definition of $\lw$ does not depend on the basepoint $b$.
\end{definition}

\begin{definition}
\label{def:mf}
Let $\ledge$ be a linear ordering on the set of edges of $G$, and let $H \subseteq G$ be an acyclic subgraph.
The \defn{$\w$-maximal subforest of $G$} is the subforest $H \subseteq M \subseteq G$ obtained from $G$ by deleting the $\lw$-least edge not in $H$ from each simple cycle.

This construction depends on $H$, which we refer to as the \defn{fixed subforest}, and $<$, the \defn{tiebreaker ordering}; however, $H$ and $<$ will usually be fixed and not mentioned explicitly in the remainder of this section.
The reader can take $H$ to be empty for all of the paper, except for \cref{sec:MF-borel}, where $H$ ensures relative ergodicity of the subforest $M$.
\end{definition}

All of our analysis of this subforest will be based on an abstract property it obeys, \cref{thm:mf-hyp} below, which relates the subforest to the following notion:

\begin{definition}
A subset $Y \subseteq X$ is \defn{($G$-)cycle-invariant} if whenever it contains an edge in a simple $G$-cycle, it also contains the entire cycle.

For example, for any bifurcation vertex $x \in X$ and side $C \subseteq X \setminus \{x\}$ of $x$ (cf.\ \cref{def:ends-furc}), the subset $C \cup \{x\}$ is cycle-invariant. (This also trivially holds for non-bifurcation vertices.)
\end{definition}

\begin{lemma}
\label{thm:mf-hyp}
For any $G$-connected cycle-invariant $Y \subseteq X$, if the $\w$-maximal subforest $M$ is such that $M|Y$ is disconnected, then every $M|Y$-component is $\w$-nonvanishing.
In particular, if $Y$ is $\w$-nonvanishing, then so is every $M|Y$-component.
\end{lemma}
\begin{proof}
Note the following key property of the $\w$-maximal subforest construction: if we restrict both $G$ and $H$ to a $G$-cycle-invariant set $Y$ (keeping the same relative weights $\w$ and tiebreaker $\ledge$), the maximal subforest we obtain is $M | Y$.
Thus we may assume that $Y = X$.

We will show that if $M$ is disconnected, then the \emph{$G$-edge boundary of every $M$-component $C$ is $\w$-nonvanishing}.
This will imply that $C$ is itself $\w$-nonvanishing, by local finiteness and our definition of the weight of an edge as the \emph{minimum} of the weights of the incident vertices.

Since $M$ is disconnected, there is an edge $e$ in $G \setminus M$ between $C$ and another $M$-component.
For any such edge $e$, since $e$ was deleted in $M$, it is the $\lw$-least edge not in $H$ in a simple $G$-cycle, which must thus contain another edge between $C$ and another $M$-component, which is $\gw e$ and also not in $H$.
Hence, there is a strictly $\lw$-increasing sequence $e_0 \lw e_1 \lw \dotsb$ on the $G$-edge boundary of $C$.
Passing to a subsequence (using local finiteness), we may assume these edges are pairwise disjoint (nonadjacent).
Then the endpoints of these edges in $C$ are infinitely many vertices $x_0, x_1, \dotsc$ on the inner boundary of $C$ with $\w(x_i) \ge \~\w(e_i) \ge \~\w(e_0)$, where the first inequality holds again because the weight of $e_i$ is defined to be the minimum of that of its endpoints.
Thus $C$ is $\w$-nonvanishing.
\end{proof}

\begin{observation}\label{thm:increasing_seq_in_the_cut}
Note that the above proof of \cref{thm:mf-hyp} exhibits something stronger, namely if our cycle-cutting algorithm (\cref{def:mf}) disconnects a cycle-invariant $G$-connected set $Y$ then there is a $\lw$-increasing sequence of pairwise disjoint edges on the boundary of each $M|Y$-component.
\end{observation}

In the rest of this subsection, we will prove various combinatorial properties of the $\w$-maximal subforest $M$; these proofs will only make use of \cref{thm:mf-hyp}, and not any other specific features of our construction.
We therefore make the following:

\begin{hypothesis}
Let $M \subseteq G$ be \emph{any} subforest for which \cref{thm:mf-hyp} holds.
\end{hypothesis}

One benefit of isolating this abstract property is:

\begin{observation}
\label{thm:mf-hyp-reweight}
If \cref{thm:mf-hyp} holds for $M$, then it continues to hold if we replace $\w$ by a different relative weight function $\w'$ such that every $\w$-nonvanishing subset is also $\w'$-nonvanishing.
In particular, we may take $\w' \equiv 1$, yielding that the following results also hold for unweighted ends.
\end{observation}

\begin{lemma}
\label{thm:mf-3ends}
If $(X,G)$ has a $\w$-trifurcation vertex $x$, then the $M$-component of $x$ has at least $3$ $\w$-nonvanishing $M$-ends.
\end{lemma}
\begin{proof}
For each of the at least $3$ $\w$-nonvanishing sides $C \subseteq X \setminus \{x\}$ of $x$, we have a $G$-connected cycle-invariant set $C \cup \{x\} \subseteq X$, whence the $M|(C \cup \{x\})$-component of $x$ is $\w$-nonvanishing by \cref{thm:mf-hyp}, whence $x$ is a $\w$-trifurcation in its $M$-component, which therefore has at least $3$ $\w$-nonvanishing ends (by \cref{thm:weights-furc}).
\end{proof}


\begin{lemma}
\label{thm:mf-dense}
Suppose every $\w$-nonvanishing boundary-finite $A \subseteq X$ containing a $\w$-bifurcation (of $(X,G)$) also contains a $\w$-bifurcation vertex (of $(X,G)$).
Then the canonical map $\partial^M_\w X -> \partial^G_\w X$ (\cref{def:canonical_end-map}) has dense image.
\end{lemma}
\begin{proof}
If $\partial^G_\w X \ne \emptyset$, then $\partial^M_\w X \ne \emptyset$ by \cref{thm:mf-hyp} (and \cref{thm:weights-nonvanish}); this proves the case $\abs{\partial^G_\w X} = 1$.
Now suppose $\abs{\partial^G_\w X} \ge 2$.
Then a basic open set in $\partial^G_\w X$ is given by $\partial^G_\w X \cap \^A$ for a side $A$ of a $\w$-bifurcation (\cref{thm:weights-furc}).
Let $F \subseteq A$ be finite connected and containing the inner boundary of $A$; then each side of $F$ is contained in either $A$ or $X \setminus A$, and so $F$ is a $\w$-bifurcation.
So $A$ contains a $\w$-bifurcation, hence also contains a $\w$-bifurcation vertex $x$.
At most one nonvanishing side $D \subseteq X \setminus \{x\}$ of $x$ can contain the $G$-connected set $X \setminus A$; thus at least one nonvanishing side $C$ of $x$ is disjoint from $X \setminus A$, hence contained in $A$.
So $A$ contains the nonvanishing $G$-connected cycle-invariant set $C \cup \{x\}$, which has a nonvanishing $M$-end by \cref{thm:mf-hyp} whose canonical image is in $\^A$.
\end{proof}

\begin{definition}
\label{def:cb-deriv}
The \defn{Cantor--Bendixson derivative} of a topological space $X$ is the closed subspace $X' \subseteq X$ of nonisolated points.
\end{definition}

\begin{lemma}
\label{thm:mf-3dense}
Suppose every $\w$-nonvanishing boundary-finite $A \subseteq X$ containing a $\w$-trifurcation (of $(X,G)$) also contains a $\w$-trifurcation vertex (of $(X,G)$).
Then every neighborhood of a nonisolated $\w$-nonvanishing $G$-end $\xi$ contains the canonical images of two distinct $\w$-nonvanishing $M$-ends from a single $M$-component with at least $3$ $\w$-nonvanishing $M$-ends.
In particular, if $Y$ denotes the union of $M$-components with at least $3$ $\w$-nonvanishing $M$-ends, then the canonical image of $\partial^M_\w Y$ is a dense subset of $(\partial^G_\w X)'$.

\end{lemma}
\begin{proof}
Let $\^A$ be a neighborhood of $\xi$; hence $A$ is nonvanishing.
As in the preceding lemma, we may assume that $A$ is a side of a $\w$-bifurcation.
Since $\xi$ is nonisolated in $\partial^G_\w X$, $G|A$ has infinitely many nonvanishing ends.
By applying \cref{thm:weights-furc} to three distinct nonvanishing ends of $G|A$ and clopen neighborhoods of them disjoint from the inner $G$-boundary of $A$, we get a finite connected $F \subseteq A$ containing the inner $G$-boundary of $A$ and with at least $3$ nonvanishing sides in $A$, hence also in $X$ since $F$ contains the inner $G$-boundary of $A$.
Thus $A$ contains a $\w$-trifurcation $F$, hence also contains a $\w$-trifurcation vertex $x$.
Now as in the preceding lemma, at most one nonvanishing side of $x$ can contain $X \setminus A$, hence at least two nonvanishing sides are contained in $A$, each of which has a nonvanishing $M$-end whose canonical image is in $\^A$.
\end{proof}

\subsection{For general connected graphs}

Given a connected locally finite graph $(X,G)$ with a relative weight function $\w$ and with many $\w$-nonvanishing ends, there may not be any $\w$-(bi/tri)furcation vertices.
Our goal now is to ``collapse'' enough $\w$-(bi/tri)furcation \emph{sets} into $\w$-(bi/tri)furcation \emph{vertices}, and then apply the analysis of the preceding subsection to the resulting quotient graph. 

The construction below is $\w$-homogeneous, so we present it for a genuine weight function $\w : X \to \#R^+$ instead of a relative weight function, to avoid notational complications. Formally, the construction is done for $\w^b$, where $b \in X$ is a fixed basepoint, observing that it does not depend on the choice of $b$.

\begin{definition}
\label{def:quot-fin}
Let $(X,G)$ be a connected locally finite graph, $\@F$ be a pairwise disjoint family of finite connected subsets $F \subseteq X$.
Let $X/\@F$ denote the quotient of $X$ identifying all vertices in a single $F \in \@F$; formally,
\begin{equation*}
\textstyle
X/\@F := \@F \cup \set{\{x\}}{x \in X \setminus \bigcup \@F}.
\end{equation*}
Let $G/\@F$ denote the $G$-adjacency graph on $X/\@F$: for $F, F' \in X/\@F$,
\begin{equation*}
F \mathrel{G/\@F} F'
\coloniff  \exists x \in F,\, y \in F'\, (x \mathrel{G} y).
\end{equation*}
We call $(X/\@F,G/\@F)$ the \defn{quotient graph} of $(X,G)$ by $\@F$.

Given a weight function $\w : X -> \#R^+$, let $\w_\@F : X/\@F -> \#R^+$ be $\sup_\pi \w$ as defined in \cref{thm:weights-map}, where $\pi : X ->> X/\@F$ is the quotient map; that is,
\begin{equation*}
\w_\@F(F) := \max_{x \in F} \w(x).
\end{equation*}
By \cref{thm:ends-map-conn}, $\pi$ induces a homeomorphism
\begin{equation*}
\^\pi : \partial^G X \cong \partial^{G/\@F}(X/\@F),
\end{equation*}
which by \cref{thm:weights-map} takes $\^\w : \partial^G X -> [0,\infty]$ to $\^{\w_\@F} : \partial^{G/\@F} (X/\@F) -> [0,\infty]$, thus restricts to
\begin{equation}
\label{eq:quot-fin:ends-nonvanish}
\^\pi : \partial^G_\w X \cong \partial^{G/\@F}_{\w_\@F}(X/\@F).
\end{equation}
\end{definition}

\begin{definition}
\label{def:mf-coll}
Let $(X,G)$ be a connected locally finite graph with a weight function $\w : X -> \#R^+$.
Consider the following method for choosing a family $\@F$ as above:
\begin{enumerate}[label=(\arabic*)]
\item
Let $\@F_1$ be a maximal disjoint family of $\w$-trifurcations.
\item
Let $\@F_2$ be a maximal disjoint set of $\w$-bifurcations in $X \setminus \bigcup \@F_1$ and $\@F_3$ a maximal disjoint set of (unweighted) bifurcations in $X \setminus \bigcup (\@F_1 \cup \@F_2)$.
\item \label{def:mf-coll:2furc}
Finally, set $\@F=\@F_1\cup\@F_2\cup\@F_3$.
\end{enumerate}

Let $M_\@F \subseteq G/\@F$ be a $\w_\@F$-maximal subforest constructed according to \cref{def:mf}, with respect to some (unspecified) fixed subforest $H \subseteq G/\@F$ and tiebreaker linear ordering $\ledge$ on the undirected $G/\@F$-edges.

Finally, let $M \subseteq G$ be a subgraph defined by arbitrarily choosing a spanning tree on each (bi/tri)furcation $F \in \@F$, and for each $M_\@F$-edge between two different $F, F' \in X/\@F$, arbitrarily choosing a single $G$-edge between them (which exists by the definition of $G/\@F \supseteq M_\@F$).

It is easily seen that $M$ is then a forest, and that $M_\@F = M/\@F$ (with each $F \in \@F$ an $M$-tree).
The respective spaces of $\w$-nonvanishing ends are related as follows:
\begin{equation}
\label{diag:mf-coll}
\begin{tikzcd}
\partial^M_\w X
\ar[d,"\^\pi^M"',"\cong"{sloped}]
\ar[r,"\^\iota"] &
\partial^G_\w X \ar[d,"\^\pi^G","\cong"'{sloped}]
\\
\partial^{M/\@F}_\w(X/\@F) \ar[r,"\^{\iota/\@F}"] &
\partial^{G/\@F}_\w(X/\@F)
\end{tikzcd}
\end{equation}
Here, the horizontal maps are the canonical maps induced by the subgraph inclusions $\iota : (X,M) -> (X,G)$ and $\iota/\@F : (X/\@F,M/\@F) -> (X/\@F,G/\@F)$ (which preserve $\w$-nonvanishing ends by \cref{thm:weights-map-nonconn}), while the vertical homeomorphisms are induced by the quotient map $\pi : X ->> X/\@F$ as in \cref{eq:quot-fin:ends-nonvanish}.
Since clearly $(\iota/\@F) \circ \pi = \pi \circ \iota$, this square commutes.
\end{definition}

We now have the following main result, summarizing the end-preservation properties of the maximal subforest construction for a single connected graph:

\begin{theorem}
\label{thm:mf-coll}
Let $(X,G)$ be a connected locally finite graph with positive weight function $\w : X -> \#R^+$.
The ``collapsed maximal subforest'' $M \subseteq G$ constructed in \cref{def:mf-coll} has the following properties, where $\iota : (X,M) -> (X,G)$ is the inclusion:
\begin{enumerate}[label=(\alph*)]
\item \label{thm:mf-coll:dense}
$\^\iota : \partial^M X -> \partial^G X$ has dense image, as does its restriction
$\^\iota : \partial^M_\w X -> \partial^G_\w X$.
\item \label{thm:mf-coll:3ends}
If $G$ has at least $3$ $\w$-nonvanishing ends, then so does at least one component of $M$.
\item \label{thm:mf-coll:3dense}
Every neighborhood of a nonisolated $\w$-nonvanishing $G$-end contains the canonical images of at least $2$ distinct $\w$-nonvanishing $M$-ends from a single $M$-component with at least $3$ $\w$-nonvanishing $M$-ends.
\end{enumerate}
\end{theorem}
\begin{proof}
By \cref{thm:mf-3ends,thm:mf-dense,thm:mf-3dense}, $\^{\iota/\@F}$ has the claimed properties, given our choice of $\@F$ in \cref{def:mf-coll}; hence so does $\^\iota$ since the above square commutes.
(To see the first part of \cref{thm:mf-coll:dense}, apply \cref{thm:mf-dense} with $\w$ replaced by the constant function $1$ (as noted in \cref{thm:mf-hyp-reweight}); the hypothesis of \cref{thm:mf-dense} is still satisfied, by \cref{def:mf-coll}\cref{def:mf-coll:2furc}.)
\end{proof}

One might expect that the properties stated in \cref{thm:mf-coll} can be strengthened in various ways; for instance, perhaps one could demand more of the $M$-components than merely ``at least $3$ nonvanishing ends''.
Indeed, we will show below that more can be said for almost every component of a quasi-pmp graph (see \cref{thm:mf-qpmp}).
However, the following shows that there are limitations to such strengthenings.

\begin{example}[Windmill graph]
\label{ex:windmill}
Let $(X,G)$ be the graph depicted in \cref{fig:windmill}.
Each ``blade'' of the windmill is a quadrant of the square lattice graph on $\#Z^2$.
The weight function $\w$ is constant $1$; thus all ends are nonvanishing.
The big dot vertices are trifurcations, and already form a maximal disjoint family of bifurcations $\@F$ as in \cref{def:mf-coll}; thus there is no need to collapse.
The tiebreaker linear ordering $\ledge$ is chosen so that each ``row'' of dotted edges is strictly increasing, and each dotted edge is $\ledge$ each solid edge.
Then the solid edges are precisely those that belong to the maximal subforest $M$.

Now the original end space $\partial^G X$ of this graph is perfect (has no isolated points).
But each $M$-component is just $3$ rays joined at their basepoint, hence has exactly $3$ ends.
This shows that \cref{thm:mf-coll}\cref{thm:mf-coll:3dense} is best possible in some sense.
Moreover, by removing some of the ``blades'' from $G$, we can cause $M$ to have infinitely many $2$-ended components, thereby showing that \cref{thm:mf-coll}\cref{thm:mf-coll:3ends} cannot be strengthened to ``every component of $M$''.
\end{example}

\begin{figure}[htb]
\centering
\includegraphics{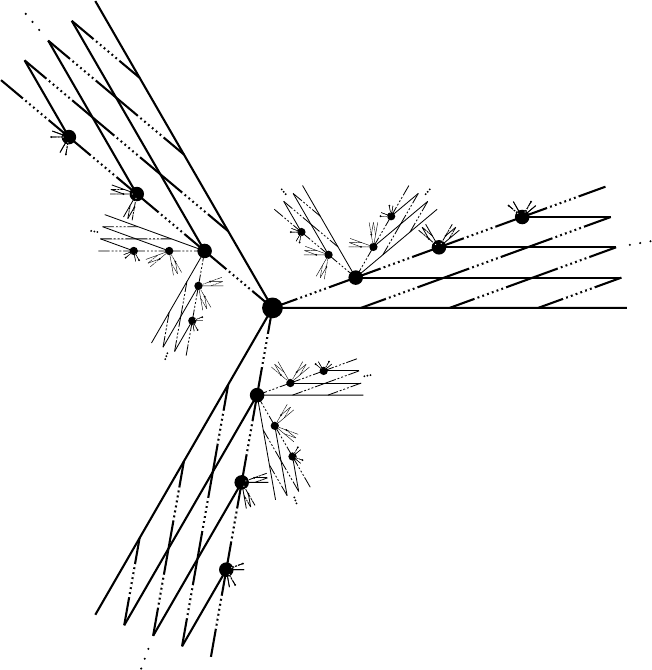}
\caption{Windmill graph described in \cref{ex:windmill}}
\label{fig:windmill}
\end{figure}

\subsection{For Borel and quasi-pmp graphs}\label{sec:MF-borel}

Let $(X,G)$ be a locally finite Borel graph equipped with a Borel cocycle $\w : G -> \#R^+$.
We recall from \cref{rmk:borel-nonconn} that $\partial^G X$, $\partial^G_\w X$, etc.\ are interpreted as the (uncountable) disjoint unions of the end spaces of all components.

\begin{theorem}
\label{thm:mf-borel}
Let $(X,G)$ be a locally finite Borel graph, $\w : G -> \#R^+$ be a Borel cocycle.
There is a Borel subforest $M \subseteq G$ with the following properties, where $\iota : (X,M) -> (X,G)$ is the inclusion:
\begin{enumerate}[label=(\alph*)]
\item \label{thm:mf-borel:dense}
The induced $\^\iota : \partial^M X -> \partial^G X$ has dense image, as does its restriction
$\^\iota : \partial^M_\w X -> \partial^G_\w X$.
\item \label{thm:mf-borel:3ends}
Each $G$-component $C \in X/G$ with $\ge 3$ nonvanishing $G$-ends contains at least one $M$-component with $\ge 3$ nonvanishing $M$-ends. 
\item \label{thm:mf-borel:3dense}
For every nonisolated nonvanishing $G$-end $\xi$, every clopen neighborhood $\^A$ of $\xi$ contains the canonical image of at least two distinct nonvanishing $M$-ends from a single $M$-component with at least $3$ nonvanishing $M$-ends.
\end{enumerate}
\end{theorem}
\begin{proof}
This follows from implementing the algorithm of \cref{def:mf-coll} in a Borel manner on each $G$-component.
In detail, the maximal family $\@F$ of trifurcations in that algorithm may be chosen in a Borel manner (see \cite[7.3]{KMtopics}), since the notions of ``$\w$-trifurcation'', etc., are clearly Borel.
This yields a finite, hence smooth, Borel subequivalence relation ${\sim_\@F} \subseteq \#E_G$, whose standard Borel quotient $X/{\sim_\@F}$ yields on each $G$-component the quotient $X/\@F$ from \cref{def:quot-fin}.

Let $Y \subseteq X$ be a Borel transversal for $\sim_\@F$, choosing from each $F \in X/\@F$ a single element with maximum $\w$-weight (i.e., maximum $\w^x$-weight for any $x \in F$).
Define now the cocycle $\w_\@F$ on $G/\@F \subseteq (X/\@F)^2$, by identifying $X/\@F$ with $Y$ and then taking the restriction of $\w : \#E_G -> \#R^+$ to $Y$.
In other words, for $F, F' \in X/\@F$, we define $\w_\@F(F,F')$ to be $\w(x,y)$ for $\w$-heaviest elements $x \in F$ and $y \in F'$.
Then for $\w$-heaviest $x$ in $F$, the weight function $\w_\@F^F : [F]_{G/\@F} -> \#R^+$ will be exactly the quotient weight function $(\w^x)_\@F(F)$ from \cref{def:quot-fin}.

So we have defined a quotient Borel graph $(X/\@F,G/\@F)$ with cocycle $\w_\@F : G/\@F -> \#R^+$, which on each $G$-component is exactly the quotient graph from \cref{def:quot-fin}.
We may now construct the $\w_\@F$-maximal subforest $M/\@F \subseteq G/\@F$ in a Borel manner as in \cref{def:mf} (with any Borel subforest $H$ of $G/\@F$, e.g., $H := \emptyset$, and any Borel tiebreaker linear ordering $\ledge$ on the undirected $G/\@F$-edges).
Finally, lift $M/\@F$ to $M \subseteq G$ as in \cref{def:mf-coll}, choosing the finite spanning trees and liftings of $G/\@F$-edges in a Borel manner using Lusin--Novikov uniformization \cite[18.10]{Kcdst}.
The desired properties of this $M$ are then given by \cref{thm:mf-coll}.
\end{proof}

\begin{remark}[Containing a prescribed subforest]
The proof of \cref{thm:mf-borel} shows that we can ensure the resulting subforest $M$ contains \emph{any} given subforest $H$ of the quotient $G/\@F$, where $\@F$ is a family of $\w$-furcations.
However, the quotienting step does not allow us to ensure that $M$ contains an arbitrary (even hyperfinite) subforest of the original graph $G$. 
For instance, consider the Schreier graph of the Bernoulli shift of $\#F_2 \times \#Z_2$, and let $H$ be the Schreier graph of the action of one of the generators of $\#F_2$. 
Then, if $\@F$ is the natural perfect matching between the copies of $\#F_2$ in each orbit, taking the quotient creates cycles that contain only edges of $H$, so $M$ cannot contain $H$.
\end{remark}

For a quasi-pmp Borel graph, the above properties of the subforest $M$ may be significantly strengthened on a conull set, due to the following fact (whose analogue in percolation on a unimodular graph is \cite[Proposition 3.9]{LyonsSchramm99}):

\begin{lemma}
\label{thm:qpmp-ends-isol}
Let $(X,G)$ be a locally finite Borel graph, $\w : G -> \#R^+$ be a Borel cocycle, and $\mu$ be a $\w$-invariant probability measure on $X$.
For a.e.\ $G$-component, the space of $\w$-nonvanishing ends either has $\le 2$ elements or is perfect\footnote{\label{ftnote:perfect-not-Polish}This does not in general imply that it has continuum-many elements because the space may not be Polish; see \cref{rmk:ends-nonvanish-fsigma}.} (has no isolated points).
\end{lemma}
\begin{figure}[htb]
\begin{center}
\begin{tikzpicture}[thick, scale=0.7]

\node[Big dot] (V41) at (0,0) {};

\node[Big dot] (V51) at (1.5,-1) {};
\node[Big dot] (V52) at (1.5,1) {};
\node[Big dot] (V71) at (4,-2) {};
\node[Big dot] (V72) at (5,2.5) {};

\node[Big dot] (V81) at (2,-3) {};

\node[Big dot] (V9) at (2,2.2) {};

\draw [name path = A] plot [smooth]  coordinates {(0,0) (-1,1.2) (-2,1.8) (-6.4,2)};
\draw [name path = B] plot [smooth]  coordinates {(0,0) (-1,-1.2) (-2,-1.8) (-6.4,-2)};

\draw plot [smooth, name path =C]  coordinates {(0,0) (0.1,1.5) (0.2,2) (0.3,2.2) (0.4,2.3) (0.5,2.44) (0.6,2.54) (0.7,2.65) (1,3)};
\draw plot [smooth, name path =D]  coordinates {(0,0) (3,0.5) (4,1) (6,3)};

\draw plot [smooth, name path =E]  coordinates {(0,0) (0.1,-1.5) (0.2,-2) (0.3,-2.2) (0.4,-2.3) (0.5,-2.44) (0.6,-2.54) (0.7,-2.65) (1,-3)};
\draw plot [smooth, name path =D]  coordinates {(0,0) (3,-0.5) (4,-1) (6,-3)};
\tikzfillbetween[of=A and B,split] {pattern=north east lines, opacity=.25, inner sep=1pt};

\draw (-4,-2) node[below] {a nonvanishing isolated end};
\draw (6,0.5) node[below] {the rest of the graph};
\draw (-0.25,-0.3) node[below] {$F$};
\draw (-3.5,0.5) node[below] {$D$};
\end{tikzpicture}
\end{center}
\caption{This is an illustration of the proof of \cref{thm:qpmp-ends-isol}, where the large dots represent the maximal disjoint Borel family $\@F$ of $\w$-trifurcations.}
\label{Fig:smooth proof}
\end{figure}
\begin{proof}
Suppose some $G$-component $C$ has at least $3$ $\w$-nonvanishing ends, at least one of which is isolated (among $\w$-nonvanishing ends).
Then any such isolated end $\xi\in \partial_\w C$ belongs to a side $D \subseteq C \setminus F$ of a $\w$-trifurcation $F$ with no other $\w$-nonvanishing ends: to see this, apply \cref{thm:weights-furc} to $\xi$, any neighborhood isolating it, and two other nonvanishing ends.
Furthermore, such $D$ then cannot also contain a $\w$-trifurcation $F'$ (of $C$), or else at least two nonvanishing sides of $F'$ would be disjoint from $F$, yielding at least two nonvanishing ends in $D$.

Now take a maximal disjoint Borel family $\@F$ of $\w$-trifurcations $F \subseteq X$ which have at least one side $D$ with exactly one nonvanishing end (which is hence isolated); see \cref{Fig:smooth proof}.
Let $Y \subseteq X$ be the union of all such sides $D$ of all $F \in \@F$.
Then each $y \in Y$ belongs to a unique such $D$ for a unique $F \in \@F$, since if it also belonged to a one-ended side $D'$ of another $F' \in \@F$, then either $F \subseteq D'$ or $F' \subseteq D$ which is impossible as noted above.
Let $E \subseteq \#E_G$ be the equivalence relation on $Y$ whose classes are exactly all such sides $D$ of $F \in \@F$, hence are nonvanishing.
Then $E$ is smooth, since we may choose in a Borel way a nonempty finite subset of each class $D \in Y/E$, namely the inner boundary of $D$ (i.e., the vertices adjacent to $F$).
So by \cref{thm:cber-compress}, $\mu(Y) = 0$, and hence $\mu([Y]_G) = 0$ by quasi-invariance of $\mu$.
But by maximality of $\@F$, $[Y]_G$ is precisely the union of the $G$-components with at least $3$ $\w$-nonvanishing ends, at least one of which is isolated (among $\w$-nonvanishing ends).
\end{proof}

\begin{corollary}
\label{thm:mf-qpmp}
The subforest $M \subseteq G$ from \cref{thm:mf-borel} additionally obeys the following for every $\w$-invariant probability measure $\mu$:
\begin{enumerate}
\item[($*$)]
For a.e.\ $G$-component $C \in X/G$ with $\ge 3$ nonvanishing $G$-ends, the space $\partial^M_\w D$ of nonvanishing ends in each $M$-component $D \subseteq C$ is nonempty and perfect\cref{ftnote:perfect_space}.
\end{enumerate}
\textcolor{black}{Moreover, for a fixed $\w$-invariant probability measure $\mu$, the subforest $M$ can be made $\mu$-ergodic relative\cref{ftnote:relative_ergodicity} to $G$.}
\end{corollary}

\begin{proof}
To conclude ($*$) notice that by \cref{thm:qpmp-ends-isol}, the union of all $M$-components with an isolated nonvanishing end is $\mu$-null.

As for relative ergodicity, let $\@F$ and $Y\subseteq X$ be as in the proof of \cref{thm:mf-borel}. 
Identifying $X/\@F$ with $Y$, we equip $X/\@F$ with the measure $\mu|_Y$. 
Note then that the Radon--Nikodym cocycle of $\#E_{G / \@F}$ with respect to $\mu |_Y$ is exactly $\w_\@F$ as in the proof of \cref{thm:mf-borel}.
Now, in the proof of \cref{thm:mf-borel}, when constructing the maximal subforest $M/\@F \subseteq G/\@F$ as in \cref{def:mf}, take the fixed subforest $H$ to be a hyperfinite subforest of $G/\@F$ that is $\mu|_Y$-ergodic relative to $G/\@F$. 
Such a subforest exists by \cite[Theorem 1.3]{Ts:hyperfinite_ergodic_subgraph}, which yields an ergodic hyperfinite subgraph; every such subgraph contains a Borel treeing by \cite[Lemma 2.4]{Mends}.
Thus, $M/\@F$ is $\mu |_Y$-ergodic relative to $G/\@F$, which easily implies that its lift $M$ (as in the proof of \cref{thm:mf-borel}) is $\mu$-ergodic relative to $G$.
\end{proof}

This proves \cref{intro:witnessing_nonamenability}.
The main content of \cref{intro:witnessing_nonamenability} is to provide a witness to nowhere amenability of the graph $G$.
The nowhere amenability itself follows more easily:

\begin{proposition}\label{3-nonvanishing=>nonamenability} 
Let $G$ be a locally finite quasi-pmp Borel graph on a standard probability space $(X,\mu)$ and let $\w : \#E_G -> \#R^+$ be the Radon--Nikodym cocycle of $\#E_G$ with respect to $\mu$. 
If each $G$-component contains $\ge 3$ $\w$-nonvanishing ends then $G$ is $\mu$-nowhere amenable.
\end{proposition}
\begin{proof}
Suppose towards the contradiction that $G$ is amenable on an $\#E_G$-invariant Borel set of positive measure. 
By restricting to that set, we may assume that $G$ is amenable. 
Then there exists a Borel treeing $T \subseteq G$ that spans every component of $G$. 
By~\cref{Anush-Robin}, it is enough to show that each $T$-component has $\ge 3$ $\w$-nonvanishing ends, which follows immediately from \cref{thm:weights-map} applied to the identity map from $T$ to $G$. However, we also give a direct proof of this fact.

Let $C$ be a $T$-component and let $F \subseteq C$ be a $\w$-trifurcation of $G$.
Each $\w$-nonvanishing side $D$ of $F$ in $G$ still has finite boundary in $T$, so $D$ must contain a $\w$-nonvanishing end of $T | C$ (by the local finiteness of $T$).
Since $F$ has $\ge 3$ $\w$-nonvanishing sides, there are $\ge 3$ $\w$-nonvanishing ends in $T | C$.
\end{proof}


\section{Maximal forest as a random subgraph}\label{sec:MFperc}

We connect the construction of a maximal subforest of a connected locally finite graph from \cref{sec:MF} to the study of random spanning forests, in particular the Free Minimal Spanning Forest (FMSF) \cite{LPS06msf}. We show that our construction of a maximal subforest yields the natural extension of FMSF for nonunimodular graphs. We first give a brief introduction to random spanning forests and their importance in percolation theory, after which we quickly review nonunimodular graphs and present our construction of the Free Maximal Spanning Forest for such graphs, as well as its properties.

Throughout this section, let $G := (V,E)$ denote a locally finite graph. 

\subsection{Random spanning forests and percolation theory}

Classically, the \defn{Free Minimal Spanning Forest} $\FMSF(G)$ on the graph $G := (V,E)$ is a random subforest of $G$ constructed as follows:
\begin{itemize}
\item Let $\{U_e\}_{e\in E}$ be a collection of independent random variables with $\mathrm{Uniform}[0,1]$ distribution. Notice that almost surely we have $U_e\neq U_{e'}$ for each pair of distinct edges $e$ and $e'$.

\item For each cycle in $G$, delete the edge $e$ with the largest value of the label $U_e$. 
\end{itemize}
In other words, for each $e \in E$, we have $e \in \mathrm{FMSF}(G)$ if and only if each cycle containing $e$ also contains another edge $e'$ with $U_e < U_{e'}$.

The \defn{Wired Minimal Spanning Forest} $\WMSF(G)$ is constructed similarly, but bi-infinite paths are also considered to be cycles.

A \defn{bond percolation} process on $G$ is a probability measure $\textbf{P}$ on $2^E$. 
We refer to elements $\omega \in 2^E$ as \defn{configurations} and we say that an edge $e \in E$ is \defn{present} (or \defn{open}) in $\omega$ if $e \in \omega$ (we think of $\omega$ as a subset of $E$). 
The connected components of $\omega$ are called \defn{clusters}. 
Finally, for a subgroup $\Gamma \le \Aut(G)$, we say that percolation $\textbf{P}$ is $\Gamma$-\defn{invariant} if the measure $\textbf{P}$ is invariant under the diagonal action of $\Gamma$ on $G$. For $p \in [0,1]$, a bond percolation process on $G$ is called $\mathrm{Bernoulli}(p)$ if every edge is present in a configuration independently with probability $p$. 
We denote the measure associated with $\mathrm{Bernoulli}(p)$ bond percolation by $\#P_p$. 
Henceforth, we will drop the word ``bond'' as we never use other kinds of percolations.

A central interest in percolation theory is the number of infinite components in a configuration. Classically, in Bernoulli$(p)$ percolation on transitive graphs, by \cite[Theorem 3]{BSbeyond} this number is constant a.s., and can take values only in  $\{0,1,\infty\}$.
In fact, there are two phase transitions that occur at the following critical values:
\begin{align}
p_c(G)&:=\inf\{p \in [0,1] : \mathbb{P}_p(\text{there is an infinite cluster})=1\}\label{eq:pc},
\\
p_u(G)&:=\inf\{p \in [0,1] : \mathbb{P}_p(\text{there is exactly one infinite cluster})=1\}. \label{eq:pu}
\end{align}
Then $\#P_p$-a.s., a configuration $\omega \in 2^E$ contains 
\begin{itemize}
\item only finite clusters when $p \in [0, p_c)$;
\item infinitely many infinite clusters when $p\in(p_c,p_u)$;
\item a unique infinite cluster when $p\in(p_u,1]$ (see \cite{HaggstromPeres99,Schonmann99}).
\end{itemize}

The study of FMSF and WMSF is closely connected to percolation theory. 
For example, \cite[Proposition 3.6]{LPS06msf} shows that for any locally finite connected graph $G$ we have that $p_c(G) = p_u(G)$ if and only if $\FMSF(G) = \WMSF(G)$. 
Thus, these random forests are related to a famous conjecture of Benjamini and Schramm \cite[Conjecture 6]{BSbeyond}, which says that a locally finite connected quasi-transitive\footnote{A graph $G$ is called \defn{quasi-transitive} if the natural action of $\Aut{(G)}$ on it has finitely many orbits.} graph $G$ is amenable if and only if $p_c(G) = p_u(G)$.



\subsection{Unimodular and nonunimodular groups and graphs}\label{sec:nonunimod}

Recall that a locally compact group is called \defn{unimodular} if a left Haar measure is also right-invariant.
A \textit{connected} locally finite graph $G := (V,E)$ is called \defn{unimodular} if its automorphism group $\Aut(G)$ is, where $\Aut(G)$ is equipped with the topology of pointwise convergence.
We refer to \cite{BLPS99inv,LyonsBook} for a survey of unimodular automorphism groups and their significance for random subgraphs.

Let $\Gamma$ be a locally compact group acting transitively on a countable set $V$ such that the stabilizer $\Gamma_v := \{\gamma\in\Gamma \mid \gamma v=v\}$ of each $v \in V$ is compact.
In the context of percolation theory, the following framework is usually stated for a closed subgroup $\Gamma \le \Aut(G)$, but since it relies only on the transitivity of the action $\Gamma \acts V$ and not the graph structure, we present it here in a more general form.

Let $m$ be a left Haar measure on $\Gamma$.
For $x,y \in V$ we define the weight of $x$ relative to $y$ by
\begin{equation}\label{def:haarweights}
\w_\Gamma(x,y) := \w_\Gamma^y(x):=m(\Gamma_x)/m(\Gamma_y).
\end{equation} 
The map $\w_\Gamma : (x,y) \mapsto \w_\Gamma(x,y)$ is an $\#R^+$-valued cocycle on the orbit equivalence relation of the action of $\Gamma$ on $V$. 
Note that by \cite[Lemma 1.29]{woess2000random} the cocycle $\w$ is invariant under the action of $\Gamma$, i.e., for all $\gamma \in \Gamma$ and $x,y \in V$, we have
\begin{equation}\label{eq:invcoc}
\w_\Gamma^y(x)
=
\frac{m(\Gamma_x)}{m(\Gamma_y)}
=
\frac{\abs{\Gamma_x y}}{\abs{\Gamma_y x}}
=
\frac{\abs{\Gamma_{\gamma x}\gamma y}}{\abs{\Gamma_{\gamma y}\gamma x}} 
=
\frac{m(\Gamma_{\gamma x})}{m(\Gamma_{\gamma y})}
=
\w_\Gamma^{\gamma y}(\gamma x).
\end{equation}
It is also a well-known fact, proven in \cite{Trofimov}, that $\Gamma$ is unimodular if and only if $\abs{\Gamma_xy}=\abs{\Gamma_yx}$ for all $x, y \in V$ in the same $\Gamma$-orbit.
This and equation \cref{eq:invcoc} imply that if $\Gamma$ is unimodular then the function $x \mapsto m(\Gamma_x)$ is constant on each $\Gamma$-orbit. 

The Mass Transport Principle (MTP) also holds in this framework, in a form analogous to the measurable version in \cref{eq:MMTP}.

\begin{theorem}[Mass Transport Principle {\cite{BLPS99inv}}]\label{thm:MTP}
Let $\Gamma$ be a locally compact group acting transitively on a countable set $V$ with compact stabilizers, and let $\w_\Gamma$ be the $\Gamma$-invariant relative weight function on $V$ as in \eqref{def:haarweights}. 
Then for any function $f : V \times V \to [0,\infty]$ which is invariant under the diagonal action of $\Gamma$, we have that for all $x,y \in V$:
\[
\sum_{z\in V} f(x,z)=\sum_{z\in V} f(z,y)\w_\Gamma^{y}(z).
\]
\end{theorem}

This principle is usually applied to a function $f(x,y)$ which is the expectation with respect to a $\Gamma$-invariant percolation on a graph $G := (V,E)$ of a non-negative $\Gamma$-invariant (with respect to the diagonal action on all three coordinates) measurable function $F(x,y,\omega)$, where $\omega$ is a percolation configuration.
Thus in our mass transport arguments, we only describe the function $F(x,y,\omega)$ and not $f(x,y)$.

\subsection{The Random Maximal Spanning Forest}

Let $G := (V,E)$ be a countable locally finite graph and let $\w : V \to \#R^+$ be a weight function.

\begin{definition}\label{defn:MaxSF}
Let $\ledge$ be a \defn{uniformly random linear ordering} (tiebreaker) on $E$, i.e., for all $e, e' \in E$,
\[
e \ledge e' :\Leftrightarrow U_{e} > U_{e'},
\]
where $\{U_e\}_{e \in E}$ is a sequence of independent random variables with $\mathrm{Uniform}[0,1]$ distribution. Then the $\w$-maximal subforest of $G$ with the random tiebreaker $\ledge$ (as in \cref{def:mf}) is a random subforest of $G$, which we call the \textbf{Free $\w$-Maximal Spanning Forest} of $G$ and denote it by $\RandomMF_\w(G)$.
\end{definition}

In other words, for every cycle in $G$ select the set of edges that are adjacent to the vertices with the smallest $\w$-weight in that cycle, and delete among them the edge that has the largest $U_e$ associated with it. 

\begin{remark}\label{rem:MF=FMSF}
It is immediate that $\RandomMF_\w(G) = \FMSF(G)$ when $\w$ is constant $1$. Thus, $\RandomMF_\w(G)$ is a natural generalization of $\FMSF(G)$ suitable to the nonunimodular setting, where we take $\w$ to be the weight function $\w_\Gamma$ induced by a nonunimodular closed subgroup $\Gamma \leq \Aut(G)$ as in \cref{def:haarweights}.
\end{remark}

Below, we assume that the graph $G$ is \textit{connected}, fix a closed subgroup $\Gamma \leq \Aut(G)$, and let $\w := \w_\Gamma$ be defined as in \cref{def:haarweights}. We then consider $\RandomMF_{\w_\Gamma}(\omega)$ for a subgraph $\omega$ of $G$. Often, $\omega$ will itself be a random subgraph of $G$ and thus, $\RandomMF_{\w_\Gamma}(\omega)$ will have two sources of randomness: one from $\omega$ and the other from the random linear ordering $\ledge$.

Similarly to $\FMSF$, the random subforest $\RandomMF_{\w_\Gamma}(\omega)$ of a $\Gamma$-invariant percolation configuration $\omega$ on $G$ has the following properties due to the fact that the weight function $\w_\Gamma$ is $\Gamma$-invariant (see \labelcref{eq:invcoc}).

\begin{proposition}\label{factor_properties}
Let $\Gamma$ be a closed subgroup of $\Aut(G)$ that acts transitively on $G$ and let $\mathbf{P}$ be a $\Gamma$-invariant percolation process on $G$. If $\omega$ is sampled from $ \mathbf{P}$ then the random forest $\RandomMF_{\w_\Gamma}(\omega)$ is a $\Gamma$-equivariant factor of $(\omega, \{U_e\}_{e \in E})$. In particular:

\begin{enumerate}[label=(\alph*)]
\item\label{thm:MFperc:inv} The distribution of $\RandomMF_{\w_\Gamma}(\omega)$ is invariant under the action of $\Gamma$.

\item \label{part:weak_mixing} If ${\bf P}$ is ergodic (resp.~weakly mixing) then $\RandomMF_{\w_\Gamma}(\omega)$ is ergodic (resp.~weakly mixing) under the action of $\Gamma$.

\end{enumerate}
In particular, all this applies to $\RandomMF_{\w_\Gamma}(G)$ because taking $\mathbf{P} := 1_E$, we have $\omega = G$ a.s.
\end{proposition}

\begin{proof}
Since $\w_\Gamma$ is $\Gamma$-invariant, and every cycle in a configuration $\omega$ equipped with the tiebreaker $\ledge$ is the same as its $\gamma$-image for every $\gamma \in \Gamma$, the map $(\omega, \{U_e\}_{e \in E}) \mapsto \RandomMF_{\w_\Gamma}(\omega)$ is $\Gamma$-equivariant. 
For part \labelcref{part:weak_mixing}, note that the sequence $\{U_e\}_{e \in E}$ is i.i.d.~and hence weakly mixing under the shift action of $\Gamma$, and the percolation $\mathbf{P}$ is ergodic (resp.~weakly mixing), so the product of the corresponding measures is ergodic (resp.~weakly mixing), and so is every $\Gamma$-equivariant factor. 
In particular, $\RandomMF_{\w_\Gamma}(\omega)$ is ergodic (resp.~weakly mixing).
\end{proof}

\begin{observation}
In the setting of \cref{factor_properties}, if $\#P_p$ is the Bernoulli$(p)$ percolation, $p \in [0,1]$, then the random forest $\RandomMF_{\w_\Gamma}(\omega)$ is a factor of i.i.d.
\end{observation}

The following lemma is helpful in proving a variety of statements in percolation theory, for instance the continuity of the percolation phase transition in unimodular nonamenable graphs \cite[Theorem 8.21]{LyonsBook}.

\begin{lemma}[{\cite[Lemma 7.7]{LyonsBook}}]\label{lem:MSFinperc}
Let $\mathbf{P}$ be a $\Gamma$-invariant percolation process on a graph $G$. 
If with positive probability (w.p.p.)\ there is a cluster of a $\mathbf{P}$-configuration $\omega$ with at least three ends, then the joint distribution of the pair $(\mathrm{FMSF}(\omega), \omega)$ is $\Gamma$-invariant and w.p.p.\ there is a tree in $\mathrm{FMSF}(\omega)$ that has at least three ends.
\end{lemma}

One of our main results,~\cref{intro:MF_in_deletion_percolation}, is an analog of this for the $\w_\Gamma$-maximal forest and $\w_\Gamma$-nonvanishing ends under additional assumptions. We prove the two parts of the statement separately in \cref{thm:MFpercA} and \cref{thm:MFperc:heavy}, respectively.

First, we establish necessary terminology.
As defined in \cref{sec:prelim-weights}, we say that a cluster $C$ is ($\w_\Gamma$-)\defn{heavy} if $\sum_{x\in C}\w_\Gamma^y(x)=\infty$ for some/every $y\in V$; otherwise, we call it ($\w_\Gamma$-)\defn{light}.

\begin{remark}[Heavy clusters have $\ge1$ $\w$-nonvanishing ends]\label{at_leaset_one_nonvanishing_end}
If $\mathbf{P}$ is a $\Gamma$-invariant percolation on $G$, the MTP implies that every cluster that has finitely many vertices of maximal weight is light a.s. Hence every heavy cluster contains a set of vertices whose relative weights are bounded away from zero. The local finiteness of $G$ now yields by K\H{o}nig's lemma that every heavy cluster has at least one $\w$-nonvanishing end a.s.
\end{remark}

The following statement follows from our constructions in~\cref{sec:MF}. 

\begin{corollary}[Maximal forest in percolation]\label{thm:MFpercA}

Let $G$ be a locally finite connected graph, $\Gamma$ be a transitive closed subgroup of $\Aut(G)$, and $\w_\Gamma$ be the $\Gamma$-invariant relative weight function on $V(G)$ induced by $\Gamma$ as in \eqref{def:haarweights}. 
Let $\mathbf{P}$ be a $\Gamma$-invariant percolation on $G$. 
Then for $\mathbf{P}$-a.e.\ configuration $\omega$, for every $\w_\Gamma$-heavy cluster $C \subseteq\omega$ which contains a $\w_\Gamma$-trifurcation vertex, the random forest $\RandomMF_{\w_\Gamma}(\omega)$ a.s.\ has a tree $T \subseteq C$ whose space of $\w_\Gamma$-nonvanishing ends is nonempty and perfect\cref{ftnote:perfect_space}.

\end{corollary}

\begin{proof}
By \cref{thm:mf-3ends} for every $\w_\Gamma$-heavy cluster $C$ in $\omega$ which contains a $\w_\Gamma$-trifurcation vertex, the random subforest $\RandomMF_{\w_\Gamma}(\omega)$ a.s.\ contains a tree in $C$ with at least $3$ nonvanishing ends.
It then follows that the space of $\w_\Gamma$-nonvanishing ends of such a tree must be perfect, by a mass transport argument analogous to those in the proofs of \cref{thm:qpmp-ends-isol} and \cite[Proposition 3.9]{LyonsSchramm99}.
\end{proof}

While~\cref{thm:MFpercA} is an extension of~\cref{lem:MSFinperc} to the setting where the relative weight function is not constant, we imposed an additional assumption, namely the existence of a $\w_\Gamma$-trifurcation vertex, instead of merely assuming the existence of $\ge 3$ $\w$-nonvanishing ends. 
We do so to ensure that $\RandomMF_{\w_\Gamma}$ still contains a $\w_\Gamma$-trifurcation vertex. 
However, it is plausible that this assumption is redundant, and if a cluster $C$ has $\ge 3$ $\w$-nonvanishing ends, then a.s.~so does a connected component of $\RandomMF_{\w_\Gamma}(\omega)$ in $C$, although we were not able to verify that.  

We highlight that the proof of \cref{lem:MSFinperc} also relies on the fact that one can force (w.p.p.)\ the presence of the trifurcation vertex in $\mathrm{FMSF}(\omega)$. 
Indeed, since $\mathrm{FMSF}$ is defined purely in terms of the linear order on the edges induced by the random labels $\{U_e\}_{e\in E}$ and does not depend on the weight function $\w_\Gamma$, one can do so by restricting to the event where on a particular finite set of edges, the labels are less or greater than $1/2$. 
When the weight function is nonconstant this is not enough to get a desired $\w_\Gamma$-trifurcation vertex in $\omega$. 
Hence we outright assume the existence of such vertices.

We believe that such an assumption is not restrictive for statements where we envision \cref{thm:MFpercA} being used. 
For instance, deletion tolerance of $\mathbf{P}$ is enough to verify this property as it allows us to ``cut'' the cycles in a given $\w_\Gamma$-trifurcation, and hence force a $\w_\Gamma$-trifurcation vertex to be present in the cluster.

\begin{definition}[Insertion and deletion tolerance]\label{def:ins-del}
Given a set of configurations $A\subseteq 2^E$ and an edge $e\in E$, let $\Pi_eA=\{\omega\cup \{e\}\mid \omega\in A\}$ and $\Pi_{\neg e}A=\{\omega\setminus \{e\}\mid \omega\in A\}$. 
A bond percolation process $\mathbf{P}$ is called \defn{insertion} (resp.\ \defn{deletion}) \defn{tolerant} if
$\mathbf{P}(\Pi_eA)>0$ (resp.\ $\mathbf{P}(\Pi_{\neg e}A)>0$) for every $e \in E$ and every non-null measurable set $A \subseteq 2^E$.
\end{definition}

\begin{example}
Let $\mathbb{P}_p$ be Bernoulli percolation on $G$.
Then for every edge $e \in E$ and every measurable $A \subseteq 2^E$ we have
\begin{equation*}
\mathbb{P}_p(\Pi_eA)\geq p\mathbb{P}_p(A) \qquad\text{and}\qquad \mathbb{P}_p(\Pi_{\neg e}A)\geq(1- p)\mathbb{P}_p(A).
\end{equation*}
In particular, this implies that Bernoulli bond percolation is both insertion and deletion tolerant.
\end{example}

The following statement shows that if $\mathbf{P}$ is insertion and deletion tolerant then a stronger conclusion holds.
\begin{theorem}\label{thm:MFperc:heavy}
Let $G$ be a locally finite connected graph, $\Gamma$ be a transitive closed subgroup of $\Aut(G)$, and $\w_\Gamma$ be the $\Gamma$-invariant relative weight function on $V(G)$ induced by $\Gamma$ as in \eqref{def:haarweights}. 
Let $\mathbf{P}$ be a $\Gamma$-invariant insertion and deletion tolerant percolation on $G$ such that a.e.\ configuration $\omega$ contains a cluster with $\ge 3$ $\w_\Gamma$-nonvanishing ends. 

Then, for $\mathbf{P}$-a.e.\ configuration $\omega$, for every $\w_\Gamma$-heavy cluster $C \subseteq\omega$, the random forest $\RandomMF_{\w_\Gamma}(\omega)$ a.s.\ has a tree $T \subseteq C$ whose space of $\w_\Gamma$-nonvanishing ends is nonempty and perfect\cref{ftnote:perfect_space}.

\end{theorem}

\begin{remark}\label{thm:MFperc:bern}
Suppose that $\mathbf{P}$-a.e.\ configuration $\omega$ contains infinitely many $\w_\Gamma$-heavy clusters.
By \cref{at_leaset_one_nonvanishing_end}, every $\w_\Gamma$-heavy cluster must contain at least one $\w_\Gamma$-nonvanishing end, hence it follows from insertion tolerance that with positive probability, there is a cluster with at least three such ends.
In particular, this applies to Bernoulli$(p)$ percolation for suitable $p \in (0,1)$, where “with positive probability” can be replaced by “almost surely” due to the ergodicity of the action of $\Gamma$.
\end{remark}

\begin{proof}[Proof of \cref{thm:MFperc:heavy}]
Let $A$ be the event where $\omega$ contains a cluster with $\ge 3$ nonvanishing ends. 
Since the conclusion of the theorem is $\Gamma$-invariant, it is enough to prove that it holds on every $\Gamma$-invariant event $B \subseteq A$ w.p.p.
By the same proof as that of \cite[Lemma~3.6]{LyonsSchramm99}, the measure $\mathbf{P}$ on the event $B$ is still insertion and deletion tolerant.
Thus, restricting to $B$, we may assume without loss of generality that $B$ holds a.s. 

Since almost every $\omega$ contains a cluster with at least three $\w_\Gamma$-nonvanishing ends, it must contain a $\w_\Gamma$-trifurcation $F$.  
By countable additivity, there is a finite subset $F\subset V(G)$ such that w.p.p.\ $\omega$ has a cluster which contains $F$ as a $\w_\Gamma$-trifurcation. 
Moreover, by deletion tolerance, w.p.p.\ the induced subgraph of $\omega$ on $F$ is a tree and is connected to each side of $F$ in its cluster by a single edge.
It follows that $F$ contains a $\w_\Gamma$-trifurcation vertex in $\omega$.

We now claim that on this event, \emph{every} heavy cluster must contain $\w_\Gamma$-trifurcation vertices. 
Indeed, otherwise we can use insertion tolerance to connect some heavy clusters that do not have any such trifurcation vertices to those that do and apply a mass transport scheme, where every vertex distributes a unit mass between the (necessarily finitely many) closest $\w_\Gamma$-trifurcation vertices in the same cluster, which yields a contradiction. 

Finally, applying \cref{thm:MFpercA} to this invariant event yields the desired conclusion.
\end{proof}

Comparing conclusions of \cref{thm:MFpercA,thm:MFperc:heavy}, as well as the properties of $\FMSF$ \cite{timar2006ends,Timar18ind}, naturally leads to \cref{q:everytree,q:indist}.

\section{Applications}\label{sec:appl}
In this section we present several concrete applications of our results.

\subsection{Coinduced actions}\label{subsec:coinduced_actions}

Let $\Gamma \le \Delta$ be countably infinite groups and let $\Gamma \acts X$ be a Borel action on a standard Borel space $X$.
Let $X^\Delta_\Gamma$ be the set of all $\Gamma$-equivariant maps from $\Delta$ to $X$, where $\Gamma$ acts on $\Delta$ by left translation.
Let $\Delta$ act on $X^\Delta_\Gamma$ (on the left) by right shift, namely, for any $\pi \in X^\Delta_\Gamma$ and $\delta, \delta' \in \Delta$, we set
\begin{equation*}
(\delta \cdot \pi)(\delta') := \pi(\delta'  \delta).
\end{equation*}
It is straightforward to verify that $\delta \cdot \pi$ is a $\Gamma$-equivariant map, so the action is well-defined.

An isomorphic description of this action may be given as follows.
Note that for any $\delta \in \Delta$, the values of a $\Gamma$-equivariant $\pi : \Delta -> X$ on the right coset $\Gamma \delta$ are uniquely determined by $\pi(\delta)$.
Thus for any family of coset representatives $(\delta_C \in C)_{C \in \Gamma\bs\Delta}$, we have a bijection
\begin{equation}
\label{eq:coind-coset}
\begin{aligned}
X^\Delta_\Gamma &\cong X^{\Gamma\bs\Delta} 
\\
\pi &|-> (\~\pi : C |-> \pi(\delta_C)).
\end{aligned}
\end{equation}
Transferring the action of $\Delta$ from $X^\Delta_\Gamma$ to $X^{\Gamma \bs \Delta}$, we get
\begin{equation}
\label{eq:coind-coset-action}
(\delta \cdot \~\pi)(C) := (\delta_C \delta) \cdot \delta_{C \cdot \delta}^{-1} \cdot \~\pi(C \cdot \delta),
\end{equation}
for each $\delta \in \Delta$, $\~\pi \in X^{\Gamma \bs \Delta}$ and $C := \Gamma \delta_C \in \Gamma \bs \Delta$.

Using such a bijection \cref{eq:coind-coset}, we may put a measure on $X^\Delta_\Gamma$ by transferring the product measure $\mu^{\Gamma\bs\Delta}$ on $X^{\Gamma\bs\Delta}$ for any base measure $\mu$ on $X$.

\begin{remark}
If $\mu$ is a $\Gamma$-invariant measure on $X$, then the measure on $X^\Delta_\Gamma$ induced in this way does not depend on the choice of coset representatives $(\delta_C)_C$, and is $\Delta$-invariant; see \cite{KQcoind}.
However, the same does not hold for quasi-invariant measures, which are our main interest.
\end{remark}

In spite of this remark, we will denote the measure on $X^\Delta_\Gamma$ induced as above by a measure $\mu$ on $X$ by $\mu^\Delta_\Gamma$, leaving the choice of coset representatives to be implied by the context.

\begin{lemma}
If $\Gamma$ acts freely on $X$, and $\mu$ is an atomless probability measure on $X$, then there is a $\mu^\Delta_\Gamma$-conull $\Delta$-invariant subset $Y \subseteq X^\Delta_\Gamma$ on which $\Delta$ acts freely.
\end{lemma}
\begin{proof}
Suppose $\pi \in X^\Delta_\Gamma$ and $\delta \in \Delta$ such that $\delta \cdot \pi = \pi$.
We have a $\Gamma$-equivariant map $X^\Delta_\Gamma -> X$, namely the projection $\pi |-> \pi(1)$; thus if $\delta \in \Gamma$, then since $\Gamma \acts X$ is free, $\delta = 1$.
If $\delta \not\in \Gamma$, then we have
$\pi(\gamma\delta) = (\delta \cdot \pi)(\gamma) = \pi(\gamma)$ for all $\gamma \in \Gamma$,
whence in particular, there are two distinct cosets $C \ne D \in \Gamma\bs\Delta$ (namely $C := \Gamma\delta$ and $D := \Gamma$) such that $\pi(C) = \pi(D) \subseteq X$.
The set $Z \subseteq X^\Delta_\Gamma$ of all $\pi$ for which there exist such $C \ne D$ with $\pi(C) = \pi(D)$ is clearly $\Delta$-invariant, and it is contained in the set of $\pi$ such that there exist $C \ne D$ with $\pi(\delta_C) \in \pi(D) = \Gamma \cdot \pi(\delta_D)$, which is null since its image under the bijection \cref{eq:coind-coset} is a countable union of diagonals, which is null since $\mu$ and all its $\Gamma$-translates is atomless.
Thus $Y := X^\Delta_\Gamma \setminus Z$ works.
\end{proof}

We are particularly interested in the case $\Delta = \Gamma * \Lambda$ for another countable group $\Lambda$.
In this case, there is a canonical choice of coset representatives $(\delta_C)_{C \in \Gamma \bs (\Gamma * \Lambda)}$, namely those elements of $\Gamma * \Lambda$ whose normal form does not start with a nonidentity element of $\Gamma$.
Note that the right translation actions of $\Gamma$ and $\Lambda$ on $\Gamma \bs (\Gamma * \Lambda)$ affect these coset representatives as follows: for $\lambda \in \Lambda$, $\gamma \in \Gamma$, and $C \in \Gamma \bs (\Gamma * \Lambda)$,
\begin{align*}
\delta_{C \cdot \lambda} &= \delta_C \lambda, \\
\delta_{C \cdot \gamma} &= \begin{cases}
\delta_C \gamma &\text{if $\delta_C \ne 1$, i.e., $C \ne \Gamma$}, \\
1 &\text{if $\delta_C = 1$, i.e., $C = \Gamma$}.
\end{cases}
\end{align*}
Thus, the formula \cref{eq:coind-coset-action} for the action of $\Delta$ on $X^{\Gamma\bs\Delta}$ becomes
\begin{align*}
(\lambda \cdot \~\pi)(C) &:= \~\pi(C \cdot \lambda), \\
(\gamma \cdot \~\pi)(C) &:= \begin{cases}
\~\pi(C \cdot \gamma) &\text{if $\delta_C \ne 1$, i.e., $C \ne \Gamma$}, \\
\gamma \cdot \~\pi(\Gamma) &\text{if $\delta_C = 1$, i.e., $C = \Gamma$}.
\end{cases}
\end{align*}
Using this, we have

\begin{lemma}
If $\mu$ is a $\Gamma$-quasi-invariant measure on $X$, with Radon--Nikodym cocycle $\w : \#E_\Gamma -> \#R^+$, then $\mu^{\Gamma \bs (\Gamma * \Lambda)}$ is a $(\Gamma*\Lambda)$-quasi-invariant measure on $X^{\Gamma \bs (\Gamma * \Lambda)}$ (with the above action), with Radon--Nikodym cocycle $\w'$ defined on generators $\lambda \in \Lambda$ and $\gamma \in \Gamma$ by
\begin{align*}
\w'(\~\pi, \lambda \cdot \~\pi) &:= 1, \\
\w'(\~\pi, \gamma \cdot \~\pi) &:= \w(\~\pi(\Gamma), \gamma \cdot \~\pi(\Gamma)).
\end{align*}
In particular, the action of $\Lambda$ on $X^{\Gamma \bs (\Gamma * \Lambda)}$ is $\mu^{\Gamma \bs (\Gamma * \Lambda)}$-preserving.
\end{lemma}
\begin{proof}
$\lambda$ acts via right shift $X^{\Gamma \bs (\Gamma * \Lambda)} -> X^{\Gamma \bs (\Gamma * \Lambda)}$, which preserves the product measure; while $\gamma$ acts via the composite of right shift followed by acting on the $\Gamma$ coordinate via $\gamma : X -> X$, the latter of which clearly has Radon--Nikodym cocycle $\w(\~\pi(\Gamma), \gamma \cdot \~\pi(\Gamma))$.
\end{proof}

\begin{example}\label{coinduced-action_nonamenable}
Let $\Gamma, \Lambda$ be infinite finitely generated groups, such that $\Lambda$ is Kazhdan, and let $\Gamma \acts (X,\mu)$ be a free quasi-pmp action.
For example, we may take $\Lambda := \SL_3(\#Z)$, $\Gamma := \#F_2$, and $X$ to be the boundary action; see \cref{F2actsamenable}.
By the above, we get an a.e.\ free quasi-pmp coinduced action $\Gamma * \Lambda \acts (X^{\Gamma \bs (\Gamma * \Lambda)}, \mu^{\Gamma \bs (\Gamma * \Lambda)})$.
Since $\Lambda$ is nonamenable and its action on $X^{\Gamma \bs (\Gamma * \Lambda)}$ is pmp by the above lemma, $\#E_\Lambda$ and hence also $\#E_{\Gamma * \Lambda}$ is nowhere amenable.
However, it is also nowhere treeable, by \cite{AS(T)}.
Using our construction, we may produce a subforest of $\#E_{\Gamma * \Lambda}$ witnessing its nonamenability: namely, since the action of $\Lambda$ is pmp, it is easily seen that each $\Lambda$-orbit yields a distinct $\w'$-nonvanishing end in the Schreier graph with respect to a union of finite generating sets for $\Gamma, \Lambda$, which by \cref{thm:mf-qpmp} contains a subforest $F$ such that for a.e.~$F$-connected component, the space of $\w'$-nonvanishing ends is nonempty and perfect.
\end{example}


\subsection{Cluster graphings for nonunimodular graphs}\label{subsec:cluster_graphings}

This subsection leverages the interplay between percolation theory and measured graph theory.
In particular, countable graphs and (uncountable) Borel graphs both appear in the same discussion.
For clarity, we denote countable graphs by regular letters $G, H$ and Borel graphs on probability spaces by calligraphic ones $\@G, \@H$.

\subsubsection{General construction}\label{subsubsec:construction_cluster_graphing}

Here we present a general construction of measurable structures on a probability space from random structures on a countable set $V$, which, in the case of graphs, is known as the cluster graphing construction.
This construction in the purely Borel setting for deterministic first-order structures was introduced in \cite[Section 3.2]{JKLcber}, while a special case was mentioned in \cite[1.6.2]{Adams:trees_amenability}.
In the measurable setting, for random graphs (specifically, bond percolation), it was developed by Gaboriau in \cite[Sections 2.2 and 2.3]{Gaboriau05}, with details worked out for unimodular graphs.
Here we present this construction for random graphs (as opposed to general first-order structures) to keep the notation and the terminology light, but we tailor it to the nonunimodular/quasi-pmp setting.

As in \cref{sec:nonunimod}, we let $\Gamma$ be a locally compact group acting transitively on a countable set $V$ with compact stabilizers.
A primary example of $\Gamma$ is a closed subgroup of $\Aut(G)$ for a locally finite connected graph $G$ on $V$.
Consider a free pmp action of $\Gamma$ on a standard probability space $(X,\mu)$.
For example, letting $\mathrm{PPP}(\Gamma)$ be the space of infinite closed discrete subsets of $\Gamma$ equipped with a Poisson point process measure $\mu$, the (left) translation action of $\Gamma$ on $\mathrm{PPP}(\Gamma)$ is free on a conull $\Gamma$-invariant Borel subset.

We extend the action $\Gamma \acts V$ to the diagonal action of $\Gamma \acts X \times V$.
Then the quotient $(X \times V)/\Gamma$ of $X \times V$ by this action is standard Borel by the Becker--Kechris theorem \cite{becker1996descriptive}. 
Indeed, there is a Polish topology such that the action of $\Gamma$ on $X$ is continuous. 
Since the stabilizer $\Gamma_o$ of any fixed vertex $o \in V$ is compact, the action of $\Gamma_o$ on $X \times \{o\}$ admits a Borel transversal $X_o \subseteq X \times \{o\}$ since every $\Gamma_o$-orbit is compact. 
Then $X_o$ is a Borel set of single representatives from every $\Gamma$-orbit of $X \times V$, so we identify it with the quotient $(X \times V) / \Gamma$.
We denote by $\pi_o : X \to X_o$ the natural surjection defined by $x \mapsto [(x,o)]_\Gamma$.
We now equip $X_o \cong (X \times V) / \Gamma$ with the pushforward measure $\mu_o := (\pi_o)_* \mu$.

Define an equivalence relation $\@E_V$ on $(X \times V) / \Gamma$ by declaring $[(x,u)]_\Gamma$ and $[(y,v)]_\Gamma$ equivalent if $y = \gamma x$ for some $\gamma \in \Gamma$.
Viewing $X_o$ as a subset of $X$, it is straightforward to check that $\@E_V$ is simply the restriction of the $\Gamma$-orbit equivalence relation $\#E_\Gamma$ on $X$ to the complete section $X_o$, which in turn is identified with $\#E_\Gamma / \Gamma_o$, i.e., 
\[
\@E_V = \#E_\Gamma |_{X_o} = \#E_\Gamma / \Gamma_o.
\]
In particular, $\@E_V$ is ergodic exactly when the $\Gamma$-action on $(X,\mu)$ is.

Note that for each $x \in X$ and $u \in V$,
\[
\big[[(x,u)]_\Gamma\big]_{\@E_V} = \{[(x,v)]_\Gamma : v \in V\} = [\pi_o(x)]_{\@E_V};
\]
in particular, $\@E_V$ is a CBER.
Furthermore, for each $x \in X$, the map 
\begin{equation}\label{eq:phi_bijection}
\begin{aligned}
\phi_x : [\pi_o(x)]_{\@E_V} &\to V
\\
[(x,v)]_\Gamma &\mapsto v
\end{aligned}
\end{equation}
is a well-defined bijection which maps $\pi_o(x)$ to the root $o \in V$. 
Indeed, well-definedness follows from the freeness of $\Gamma \acts X$ because if $[(x,u)]_\Gamma = [(x,v)]_\Gamma$ then $u = v$.
Moreover, the map $x \mapsto \phi_x$ is $\Gamma$-equivariant, more precisely:
\begin{equation}\label{eq:phi_is_equivariant}
\phi_{\gamma x} (\~y) = \gamma \phi_x (\~y)    
\end{equation}
for all $x \in X$, $\~y \in [\pi_o(x)]_{\@E_V}$, and $\gamma \in \Gamma$ because we may write $\~y = [(x,v)]_\Gamma = [(\gamma x, \gamma v)]_\Gamma$ for some $v \in V$, and hence $\phi_{\gamma x}(\~y) = \phi_{\gamma x}([(\gamma x, \gamma v)]_\Gamma) = \gamma v = \gamma \phi_{x}(\~y)$.
Note that $\phi_x$ truly depends on $x$ and not just on $\pi_o(x)$, since if $\gamma \in \Gamma_o$ acts nontrivially on $V$, then $\pi_o(x) = \pi_o(\gamma x)$ but $\phi_x$ is different from $\phi_{\gamma x}$ by \labelcref{eq:phi_is_equivariant}.

The following lemma was essentially proven in \cite[Theorem 2.5]{Gaboriau05}.

\begin{lemma}\label{Radon-Nikodym_of_E_V}
The equivalence relation $\@E_V$ is quasi-pmp with respect to $\mu_o$ and the corresponding Radon--Nikodym cocycle $\w_{\mu_o} : \@E_V \to \#R^+$ is given by
\begin{equation*}
\w_{\mu_o}^{[x,u]_\Gamma}([x,v]_\Gamma)
:=
\w_\Gamma^u(v),
\end{equation*} 
where $\w_\Gamma$ is as in \eqref{def:haarweights}.
In particular, $\@E_V$ is pmp if and only if $\Gamma$ is unimodular.
\end{lemma}

\begin{proof}
For any pair of points $\~x,\~{y}\in X_o$ we may choose representatives $[(x,u)]_\Gamma$, $[(x,v)]_\Gamma$ to agree on the first coordinate. 
Thus the function $\w_{\mu_o}$, as in the statement of the lemma, is well-defined since $\w_\Gamma$ is invariant under the diagonal action of $\Gamma$.
Because $\w_\Gamma$ is a cocycle (relative weight function), so is $\w_{\mu_o}$.

It thus remains to check that $\w_{\mu_o}$ satisfies the measurable mass transport principle, as in \eqref{eq:MMTP}.
Let $F : X_o^2 \to [0,\infty]$ be a Borel function.
Again because we can make the representatives of any two points in $X_o$ agree on the first coordinate, we can rewrite $F$ as $\widehat{F} : V^2 \times X\to [0,\infty]$ in the following way 
\[
\^{F}(u,v;x) := F\left([(x,u)]_\Gamma,[(x,v)]_\Gamma\right).
\]
Notice that the map $(u,v) \mapsto \int_X \widehat{F}(u,v;x)d\mu(x)$ is invariant under the diagonal action of $\Gamma$, so we may apply the original mass transport principle (\cref{thm:MTP}) to it.
Thus,
\begin{align*}
\int_{X_o} \sum_{\~{y} \in [\~x]_{\@E_V}} F(\~x,\~{y}) d\mu_o(\~x) 
&=
\int_X \sum_{v \in V} F([x,o]_\Gamma,[x,v]_\Gamma) d\mu(x)
=
\sum_{v \in V} \int_X \widehat{F}(o,v;x) d\mu(x)
\\
\textnormal{[by the MTP]}\quad
&=
\sum_{v \in V} \int_X \widehat{F}(v,o;x) \w_\Gamma^o(v)d\mu(x)
\\
&=
\int_X \sum_{v \in V} F([x,v]_\Gamma,[x,o]_\Gamma) \w_{\mu_o}^{[x,o]_\Gamma}([x,v]_\Gamma)d\mu(x)
\\
&=
\int_{X_o} \sum_{\~{y}\in [\~x]_{\~E}} F(\~{y},\~x) \w_{\mu_o}^{\~x}(\~{y})d\mu_o(\~x).
\qedhere
\end{align*}
\end{proof}

We now define a natural
bijection and its inverse between Borel structures on $\@E_V$ and $\Gamma$-invariant random structures on $V$ which are $\Gamma$-equivariant factors of $\Gamma \acts (X,\mu)$.
Here, we think of $2^{V \times V}$ as the space of graphs on $V$.

\begin{proposition}\label{graphing-and-rg} For $\Gamma$, $V$, $(X,\mu)$ as above, we have:
\begin{enumerate}[(a), leftmargin=*]
\item\label{graphing-to-rg} For every Borel graph $\@G \subseteq \@E_V$ there is a unique $\Gamma$-equivariant Borel map $\omega_\@G : X \to 2^{V \times V}$ such that for all $x \in X$, the map $\phi_x$, as in \cref{eq:phi_bijection}, is an isomorphism of relatively weighted graphs $(\@G |_{[\pi_o(x)]_{\@E_V}}, \w_{\mu_o})$ and $(\omega_\@G(x), \w_\Gamma)$; see \cref{fig:diagram}.
In particular, $(\omega_\@G)_* \mu$ is the law of the $\Gamma$-invariant random graph $\omega_\@G$.

\item\label{rg-to-graphing} For every $\Gamma$-equivariant Borel map $\omega : X \to 2^{V \times V}$, there is a unique Borel graph $\@G_\omega \subseteq \@E_V$ such that $\omega_{\@G_\omega} = \omega$.
\end{enumerate}
More generally, these statements hold for $L$-structures for any countable relational language $L$; precisely, in part \cref{graphing-to-rg} $\@G$ can be any Borel $L$-structure on $\@E_V$ in the sense of \cite{Chen-Kechris}, and in part $\cref{rg-to-graphing}$, $\omega$ can be any $\Gamma$-invariant random $L$-structure which is a factor of $\Gamma \acts (X,\mu)$.
\begin{figure}[htb]
\begin{center}
\begin{tikzpicture}[thick, scale=0.7]
\node at (-3,1) {$(X,\mu)$};
\node at (3.25,1) {$(X_o,\mu_o)$};
\node at (4.7,1) {$\supseteq$};
\node at (7.8,1) {$([\pi_o(x)]_{\@E_V}, \@G \vert_{[\pi_o(x)]_{\@E_V}})$};
\node at (0,-2.1) {$\left(2^{V\times V},\mathbf{P}\right)$};
\node at (1.8,-2.1) {$\ni$};
\node at (3.5,-2.1) {$(V, \omega_\@G(x))$};

\draw [->>] 
(-2,1) 
to node [above,sloped] {$\pi_o$} 
(2,1);
\draw[->>] 
(-3,0.5)
to
node[below,sloped] {$\omega_\@G$}
node[above,sloped] {$\Gamma$-\scriptsize{eqvar}.}
(-1,-1.5);
\draw[->] 
(6.3,0.6) 
to 
node[below,sloped] {$\varphi_x$}
node[above,sloped] {$\cong$}
(3.0,-1.6768);
\end{tikzpicture}	
\end{center}
\caption{The maps between the free pmp action $\Gamma \acts (X,\mu)$, its quotient $(X_o,\mu_o)$ by the stabilizer $\Gamma_o$ on which the Borel graph $\@G$ is defined, and the space of graphs on $V$ on which $\Gamma$ acts by shift.
}
\label{fig:diagram}
\end{figure}
\end{proposition}
\begin{proof}
For concreteness, we only prove for a Borel graph $\@G \subseteq \@E_V$ since the general case of Borel $L$-structurings is analogous.

For part \cref{graphing-to-rg}, define the map $\omega_\@G : X \to 2^{V \times V}$ by mapping $x$ to the $\phi_x$-image of $\@G |_{[\pi_o(x)]_{\@E_V}}$; in words, $(u,v)$ is present in $\omega_\@G(x)$ exactly when $[(x,u)]_\Gamma$ and $[(x,v)]_\Gamma$ are adjacent in $\@G$.
The $\Gamma$-equivariance of the map $x \mapsto \phi_x$ as in \cref{eq:phi_is_equivariant} implies that $\omega_\@G$ is $\Gamma$-equivariant; indeed, if the edge $(u,v)$ is present in $\omega_\@G(x)$, then $(\gamma u, \gamma v)$ is present in $\omega_\@G(\gamma x)$.
Thus, $\omega_\@G$ is a factor map and $(\omega_{\@G})_* \mu$ is the law of a $\Gamma$-invariant random graph on vertex set $V$. 
The rest of the statement follows from the fact that $\phi_x$ is a bijection for each $x \in X$, together with \cref{Radon-Nikodym_of_E_V}.

For part \cref{rg-to-graphing}, define the graph $\@G \subseteq \@E_V$ by declaring $[(x,u)]_\Gamma$ and $[(x,v)]_\Gamma$ adjacent in $\@G$ exactly when the vertices  $\phi_x([(x,u)]_\Gamma)$ and  $\phi_x([(x,v)]_\Gamma)$ form an edge in $\omega(x)$, i.e.\ $(u,v) \in \omega(x)$.
The $\Gamma$-equivariance of $\omega$ and $\Gamma$-equivariance of the map $x \mapsto \phi_x$ as in \cref{eq:phi_is_equivariant} implies that this is well-defined, i.e., does not depend on the choice of the representatives $(x,u)$ and $(x,v)$.
\end{proof}

\begin{definition}
Let $\mathbf{P}$ be the law of a $\Gamma$-invariant random graph on $V$.
A Borel graph $\@G$ on a standard probability space $(X_o,\mu_o)$ is called a \defn{cluster graphing of $\mathbf{P}$} if $X_o = (X \times V) / \Gamma$ and $\mu_o = (\pi_o)_* \mu$ for some free pmp action $\Gamma \acts (X,\mu)$ such that $(\omega_\@G)_* \mu = \bfP$, where $\pi_o : X \to (X \times V) / \Gamma$ is the quotient map as above and $\omega_\@G$ is as in \cref{graphing-and-rg}.
\end{definition}

\begin{corollary}\label{existence_cluster_graphing}
Every $\Gamma$-invariant random graph on $V$ with law $\mathbf{P}$ admits a cluster graphing.
\end{corollary}
\begin{proof}
This follows from \cref{graphing-and-rg}\labelcref{rg-to-graphing} and the fact that $\mathbf{P}$ is always a $\Gamma$-equivariant factor of some free pmp action $\Gamma \acts (X,\mu)$. 
For example, take as $(X,\mu)$ the space $\mathrm{PPP}(\Gamma) \times (2^{V \times V},\mathbf{P})$ with the diagonal action (technically, instead of $\mathrm{PPP}(\Gamma)$, one should take a $\Gamma$-invariant conull Borel subset of the free part), then the projection on the second coordinate is a desired factor map.
\end{proof}

For a quasi-pmp locally countable graph $\@G$ on a standard probability space $(X,\mu)$, let $X_\infty(\@G) \subseteq X$ denote the union of $\w$-infinite $\@G$-components, where $\w$ is the associated Radon--Nikodym cocycle.
Note that $X_\infty(\@G)$ depends on the choice of $\w$ and hence is defined up to a null set.
We call the restriction $\@G |_{X_\infty(\@G)}$ the \textbf{heavy part} of $\@G$.
We now adapt the definition of indistinguishability of infinite clusters from \cite{LyonsSchramm99,Penfei18}. 
We say that a $\Gamma$-invariant random graph $\omega$ on $V$ has \textbf{indistinguishable $\w_\Gamma$-heavy clusters} if for every measurable $A\subset 2^{V} \times 2^{V\times V}$ that is invariant under the diagonal action of $\Gamma$ on $V$, $\bfP$-a.s., for all $\w_\Gamma$-heavy clusters $C$ in $\omega$, we have $(C,\omega)\in A$, or for all $\w_\Gamma$-heavy clusters $C$ in $\omega$, we have $(C,\omega)\notin A$.

\begin{lemma}\label{erg_gives_indist}
Let $\mathbf{P}$ be the law of a $\Gamma$-invariant ergodic random graph on $V$.
If $\mathbf{P}$ admits a cluster graphing whose heavy part is ergodic, then $\mathbf{P}$ has indistinguishable $\w_\Gamma$-heavy clusters.
\end{lemma}
\begin{proof}
Let $\@G$ be a cluster graphing of $\bfP$ on the quotient space $(X \times V) / \Gamma$, for some free pmp action $\Gamma \acts (X,\mu)$ such that $(\omega_\@G)_* \mu = \bfP$, and suppose that the heavy part $\@G_\infty := \@G |_{X_\infty(\@G)}$ is ergodic with respect to $(\pi_o)_* \mu$.

For a configuration $\omega \in 2^{V \times V}$ and a vertex $v \in V$, we let $C_v(\omega)$ denote the cluster at $v$ in $\omega$.
Since $\bfP$ is $\Gamma$-invariant, to prove the conclusion it suffices to fix a $\Gamma$-invariant Borel set $A \subseteq 2^V \times 2^{V \times V}$ and show that the set 
\[
A_o := \left\{ \omega \in 2^{V \times V} : \w_\Gamma(C_o(\omega)) = \infty \text{ and } (C_o(\omega), \omega) \in A \right\}
\]
is $\bfP$-null or $\bfP$-conull.
Note that for each $\gamma \in \Gamma$,
\begin{equation}\label{eq:cluster-invariance of A_o}
\text{if $\omega \in A_o$, and $o$ and $\gamma^{-1} o$ form an edge in $\omega$, then $\gamma \omega \in A_o$}
\end{equation}
because
$
(C_o(\gamma \omega), \gamma \omega) 
=
(\gamma C_{\gamma^{-1} o} (\omega), \gamma \omega)
= 
(\gamma C_o (\omega), \gamma \omega)
=
\gamma (C_o(\omega), \omega) \in A
$
by the $\Gamma$-invariance of $A$.

To prove that $\omega_\@G^{-1}(A_o)$ is $\mu$-null or $\mu$-conull, it suffices to show that $\~A_o := \pi_o(\omega_\@G^{-1}(A_o))$ is $(\pi_o)_* \mu$-null or $(\pi_o)_* \mu$-conull.
It follows from the definitions of $A_o$ and $\omega_\@G$ that the set $\~A_o \subseteq X_\infty(\@G)$.
Thus, by the ergodicity of $\@G_\infty$ and the fact that $\~A_o$ is universally measurable (being an analytic set, see \cite[Theorem 21.10]{Kcdst}), it remains to show that $\~A_o$ is $\@G_\infty$-invariant.

For this, it suffices to take two $\@G$-adjacent points $\pi_o(x)$ and $\pi_o(\gamma x)$ from $(X \times V) / \Gamma$ with $x \in \omega_\@G^{-1}(A_o)$ and $\gamma \in \Gamma$, and show that $\gamma x \in \omega_\@G^{-1}(A_o)$ as well.
Now $\pi_o(x) = [(x,o)]_\Gamma$ and $\pi_o(\gamma x) = [(x, \gamma^{-1} o)]_\Gamma$, and hence $o$ and $\gamma^{-1} o$ form an edge in $\omega_\@G(x)$ by \cref{graphing-and-rg}\labelcref{graphing-to-rg}.
Thus the $\Gamma$-equivariance of $\omega_\@G$ and \labelcref{eq:cluster-invariance of A_o} yield $\omega_\@G(\gamma x) = \gamma \omega_\@G(x) \in A_o$, so $\gamma x \in \omega_\@G^{-1}(A_o)$.
\end{proof}

\subsubsection{Applications to cluster graphings}\label{subsubsec:applications_cluster_graphings}

We now give corollaries of \cref{intro:witnessing_nonamenability} in the context of percolation theory via cluster graphings.

\begin{corollary}\label{cluster-graphing_nonamenable}
Let $G$ be a locally finite connected graph, $\Gamma$ be a transitive closed subgroup of $\Aut(G)$, and $\w_\Gamma$ be the $\Gamma$-invariant relative weight function on $V$ induced by $\Gamma$ as in \eqref{def:haarweights}. 
Let $\mathbf{P}$ be a $\Gamma$-invariant percolation such that a.s.\ there is a cluster with $\ge 3$ $\w_\Gamma$-nonvanishing ends.
Then every cluster graphing $\@G$ of $\bfP$ on a standard probability space $(X_o,\mu_o)$ is $\mu_o$-nonamenable.
In fact, $\@G$ contains a relatively ergodic $\mu_o$-nonamenable Borel subforest $\@F$ with the following property: 

For a.e.\ $\@F$-component $T$, if $T$ is contained in a $\@G$-component with $\ge 3$ $\w_{\mu_o}$-nonvanishing ends, then $T$ also has $\ge 3$ $\w_{\mu_o}$-nonvanishing ends.
\end{corollary}
\begin{proof}

Using the notation introduced above, let $\@G$ be a cluster graphing on a standard probability space $(X_o,\mu_o)$  induced by percolation process $\mathbf{P}$, and let $\w_{\mu_o}$ denote the Radon--Nikodym cocycle of $\#E_{\@G}$ with respect to $\mu_o$.

Invoking the isomorphisms from \cref{graphing-and-rg}, the union of $\@G$-components with $\ge 3$ $\w_{\mu_o}$-nonvanishing ends is a positive measure set $A$.
By \cref{intro:witnessing_nonamenability}, $\@G$ restricted to $A$ contains a Borel subforest $\@F$ with a.e.\ $\@F$-component having at least three nonvanishing ends.
In particular, by \cref{Anush-Robin}, $\@F$, and hence $\@G |_A$, is $\mu_o$-nowhere amenable on $A$.
\end{proof}

We now transfer the subforest $\@F$ given by \cref{cluster-graphing_nonamenable} back to the percolation setting. 

This results in a random subforest of $G$ which is closely related to $\RandomMF_{\w_\Gamma}(G)$ and has indistinguishable heavy components, see \cref{thm:indist_forest}.
The key difference between these forests comes from the fact that the construction of $\@F$ given by \cref{intro:witnessing_nonamenability} (see the proof of \cref{thm:mf-qpmp}) prohibits the cycle-cutting procedure from cutting certain edges to make $\@F$ contain a subgraph that is relatively ergodic to $\@G$, while the cycle-cutting algorithm for $\RandomMF_{\w_\Gamma}(G)$ may cut those edges.
Thus, we see it as a motivation towards \cref{q:indist}.

\begin{corollary}[Indistinguishable ``cousin" of $\RandomMF_{\w_\Gamma}(G)$]\label{thm:indist_forest}
Let $G$ be a locally finite connected graph, $\Gamma$ be a transitive closed subgroup of $\Aut(G)$, and $\w_\Gamma$ be the $\Gamma$-invariant relative weight function on $V(G)$ induced by $\Gamma$ as in \eqref{def:haarweights}. 
Assume that $G$ has $\ge 3$ $\w_\Gamma$-nonvanishing ends.
Then $G$ admits a $\Gamma$-invariant random subforest, whose components are indistinguishable and have nonempty perfect spaces of $\w_\Gamma$-nonvanishing ends.
\end{corollary}
\begin{proof}
Let $\@G$ be a cluster graphing on a standard probability space $(X_o,\mu_o)$ of the (trivial) percolation process $\delta_{G}$, and let $\w_{\mu_o}$ denote the Radon--Nikodym cocycle of $\#E_{\@G}$ with respect to $\mu_o$. 
By \cref{graphing-and-rg}, we have that a.e.~$\@G$-component has $\ge3$ $\w_{\mu_o}$-nonvanishing ends.
Let $\@F$ be a Borel subforest of $\@G$ given by \cref{cluster-graphing_nonamenable}.

Note that $\@F$ is relatively ergodic to $\@G$ and a.e.~$\@F$-component has $\ge 3$ $\w_{\mu_o}$-nonvanishing ends. 
It then follows from \cref{erg_gives_indist} that the random spanning subforest of $G$ given by $\omega_{\@F}$, as in \cref{graphing-and-rg}, has indistinguishable components, which a.s.\ have $\ge 3$ $\w_\Gamma$-nonvanishing ends.
\end{proof}

Next we will give a concrete example of the situation in the hypothesis of \cref{cluster-graphing_nonamenable}. First, we recall another transition parameter from percolation theory defined as follows:
\begin{align}
p_h:=p_h(G,\Gamma)&:=\inf\{p\in[0,1]: \mathbb{P}_p(\text{there is a $\w_\Gamma$-heavy cluster})=1\}\label{eq:ph}.
\end{align}
Recalling $p_c(G)$ and $p_u(G)$ from \cref{eq:pc,eq:pu}, we note that \cite[Theorem 4.1.6]{Haggstrom99} yields that infinite light clusters cannot coexist with heavy ones. 
Furthermore, mass transport implies that the unique infinite cluster has to be heavy, so:
\[
p_c(G)\leq p_h(G)\leq p_u(G)\leq 1.
\]
Moreover, each of these inequalities could be strict, see \cite{Haggstrom99,Timar06nonu,Hutchcroft20} for details and examples.

\begin{example}[Free product of $\textrm{GP}(2)$ and $\#Z^2$]\label{ex:GP}
Let $\textrm{GP}(k)$ be the grandparent graph, originally introduced in \cite{Trofimov}. 
Such a graph is constructed as follows: start with a $(k + 1)$-regular tree with a distinguished end, so each vertex has a unique parent and $k$ children. 
Connect every vertex to its grandparent. 
It is easy to check that $\textrm{GP}(k)$ is nonunimodular. 
Moreover, equation \eqref{def:haarweights} implies that if $x$ is a parent of $y$ then $\w_\Gamma^x(y)=\frac{1}{k}$, where $\Gamma := \Aut(\textrm{GP}(k))$. 
Therefore, vertices $x$ and $y$ have the same weight if and only if they are in the same generation.

Let the graph $G$ be the free product%
\footnote{The free product of graphs is similar to the Cayley graph of the free product of groups.
There are several definitions of such products for graphs present in literature; we use the one in, for example, \cite[Section 4]{Pisanski02}.
This definition and various others are compared in \cite{Carter20}, where it is shown that they are all equivalent for vertex-transitive graphs.}
of the grandparent graph $\textrm{GP}(2)$ and $\#Z^2$, and consider Bernoulli$(p)$ percolation $\#P_p$ on it. 
Notice that $G$ is still nonunimodular and the induced cocycle is equal to $1$ on each edge that comes from $\#Z^2$, while it takes values $\{\frac{1}{4},\frac{1}{2},2,4\}$ on the directed edges that come from $\textrm{GP}(2)$. 
A classical result of Kesten \cite{Kesten80} states that $p_c(\#Z^2)=\frac12$. 
Thus, for every $p > \frac{1}{2}$, since the vertices in the same copy of $\#Z^2$ are of the same weight, $\#P_p$-a.e.\ percolation configuration on $G$ will contain a heavy cluster.
Therefore, $p_h(G) \leq \frac{1}{2}$. 
On the other hand, by a similar argument as for a $3$-regular tree, $p_u(\textrm{GP}(2))=1$, and hence $p_u(G) = 1$. Finally, \cref{cluster-graphing_nonamenable} yields that for every $p \in \left(\frac{1}{2}, 1\right)$ a cluster graphing  $\@G^\cl$ equipped with the measure induced by $\mathbb{P}_p(G)$ is nonamenable, and in fact, contains a nonamenable Borel subforest.
\end{example}

Witnessing nonamenability for a cluster graphing can be used to show purely graph-theoretic or percolation-theoretic properties. 
An example of such a property is $\w$-visibility of a graph introduced in \cite{Ts:hyperfinite_ergodic_subgraph}.

\begin{definition}\label{def:visibility}
Let $G := (V,E)$ be a connected graph. 
Given a relative weight function (cocycle) $\w : V^2 \to \#R^+$, we say that $y \in V$ is \textbf{$\w$-visible} from $x \in V$ if there is a $G$-path $x = x_0, x_1, x_2, \ldots, x_{k} = y$ such that $\w^x(x_i) \le 1$ for all $i \le k$.
Let $N^\w(x)$ be the set of all $y \in V$ that are $\w$-visible from $x$. 
We say that $G$ has \textbf{finite $\w$-visibility} if for all $x \in V$ the set $N^\w(x)$ is $\w$-light; otherwise, we say that $G$ has \textbf{infinite $\w$-visibility}.
\end{definition}

In \cite[Theorem 1.8]{Ts:hyperfinite_ergodic_subgraph}, it is proven that if almost every component of a quasi-pmp Borel graph $G$ has finite visibility, then $G$ is amenable.
Thus, almost every component of a nowhere amenable quasi-pmp Borel graph has infinite visibility, which enables combinatorial techniques such as mass transport.

\begin{theorem}\label{thm:visibility}
Let $G := (V,E)$ be a countable locally finite graph and let $\Gamma$ be a closed subgroup of $\Aut(G)$ that acts transitively on $G$. 
Let $\w_\Gamma$ be the relative weight function induced by $\Gamma$ as in \eqref{def:haarweights}. 
Let $p\in(p_h(G,\Gamma), p_u(G))$ and let $\#P_p$ denote the Bernoulli$(p)$ percolation on $G$.
Then:
\begin{enumerate}[label=(\alph*)]
\item\label{thm:visibility_cluster} Every $\w_\Gamma$-heavy cluster has infinite $\w_\Gamma$-visibility $\mathbb{P}_p$-a.s.

\item\label{thm:visibility_forest}
For $\mathbb{P}_p$-a.e.\ configuration $\omega$, for every $\w_\Gamma$-heavy cluster $C \subseteq\omega$, the random forest $\RandomMF_{\w_\Gamma}(\omega)$ a.s.\ has a tree in $C$ with infinite $\w_\Gamma$-visibility.
\end{enumerate}
\end{theorem}

\begin{proof}
By \cref{thm:MFperc:heavy,thm:MFperc:bern} we have that for $\#P_p$-a.e.\ configuration $\omega$:
\begin{enumerate}[(i)]
\item there are infinitely many heavy clusters in $\omega$;

\item each heavy cluster contains $\ge 3$ nonvanishing ends;

\item \label{tree_with_nonvanishing_ends} for every $\w_\Gamma$-heavy cluster $C$ the random forest $\RandomMF_{\w_\Gamma}(\omega)$ almost surely contains a tree in $C$ with $\ge 3$ $\w_\Gamma$-nonvanishing ends.
\end{enumerate}

Part \cref{thm:visibility_cluster}. By the proof of \cref{cluster-graphing_nonamenable}, the restriction of a cluster graphing $\@G$ to the union $A$ of infinite components is $\mu_o$-nowhere amenable on $A$. 
Hence, by \cite[Theorem 1.8]{Ts:hyperfinite_ergodic_subgraph} almost every infinite $\@G$-component in $A$ has infinite $\w_{\mu_o}$-visibility, where $\w_{\mu_o}$ is the Radon--Nikodym cocycle of $\#E_{\@G}$ with respect to the underlying measure $\mu_o$. 
Invoking the isomorphisms from \cref{graphing-and-rg} yields that for $\mathbb{P}_p$-a.e.\ configuration $\omega$, all of its heavy clusters have infinite $\w_\Gamma$-visibility.

Part \cref{thm:visibility_forest}. By \labelcref{tree_with_nonvanishing_ends} above, it is enough to show that a.s.\ every tree in $\RandomMF_{\w_\Gamma}(\omega)$ with $\ge 3$ $\w_\Gamma$-nonvanishing ends has infinite $\w_\Gamma$-visibility.
But this follows from part \labelcref{thm:visibility_cluster} applied to $\RandomMF_{\w_\Gamma}(\omega)$ instead of $\omega$.
\end{proof}

\bibliographystyle{alphaurl}
\bibliography{ref-qpmp.bib}

\end{document}